\documentclass[12pt,reqno]{amsart}
\usepackage[margin=1.0in]{geometry}
\usepackage{amssymb}
\newtheorem{theorem}{Theorem}[section]
\newtheorem{lemma}[theorem]{Lemma}
\newtheorem{proposition}[theorem]{Proposition}

\newcommand{\R}{\mathbb{R}}
\newcommand{\N}{\mathbb{N}}

\newcommand{\Z}{\mathbb{Z}}

\begin{document}

\title{A logistic equation with nonlocal interactions}

\author{Luis Caffarelli}
\address[Luis Caffarelli]{The University of Texas at Austin,
Department of Mathematics and
Institute for Computational Engineering and Sciences,
2515 Speedway, Austin, TX 78751, USA}
\email{caffarel@math.utexas.edu}

\author{Serena Dipierro}
\address[Serena Dipierro]{Otto-von-Guericke-Universit\"at Magdeburg,
Fakult\"at f\"ur Mathematik,
Institut f\"ur Analysis und Numerik,
Universit\"atsplatz 2, 39106 Magdeburg, Germany}
\email{serena.dipierro@ed.ac.uk}

\author{Enrico Valdinoci}
\address[Enrico Valdinoci]{Weierstra{\ss} Institut f\"ur Angewandte Analysis
und Stochastik, Mohrenstrasse 39, 10117 Berlin, Germany,
and Dipartimento di Matematica, Universit\`a degli studi di Milano,
Via Saldini 50, 20133 Milan, Italy, and
Istituto di Matematica Applicata e Tecnologie Informatiche,
Consiglio Nazionale delle Ricerche,
Via Ferrata 1, 27100 Pavia, Italy}
\email{enrico@mat.uniroma3.it}

\begin{abstract}
We consider here a logistic equation, 
modeling processes of nonlocal character both in the diffusion
and proliferation terms.

More precisely, for populations that propagate according to a 
L\'evy process and can reach resources in a neighborhood
of their position, we compare (and find explicit threshold
for survival) the local and nonlocal case.

As ambient space, we can consider:
\begin{itemize}
\item bounded domains,
\item periodic environments,
\item transition problems, where the environment consists
of a block of infinitesimal diffusion and an adjacent nonlocal one.
\end{itemize}
In each of these cases, we analyze
the existence/nonexistence of solutions in terms of the spectral properties
of the domain. In particular, we give a detailed description of the fact that
nonlocal populations may better adapt to sparse resources
and small environments.
\end{abstract}


\maketitle

\section{Introduction}

In this paper we study stationary solutions for a logistic
equation. The solution~$u$ can be interpreted, from the point of view
of mathematical biology, as the density of a population living in
some environment~$\Omega\subseteq\R^n$.

In the classical logistic equation
(see e.g.~\cite{Verhulst1845, PRE:8388973, pearl}), the population is supposed
to increase proportionally to the resource of the environment
(the growing effect being modeled by a nonnegative function~$\sigma$)
and to die when the resources get extinguished
(the dying effect being described by a nonnegative function~$\mu$).
The population is also assumed to diffuse randomly
(the random diffusion being modeled by the Laplace operator).
These considerations lead to a detailed study of the evolution equation
$$ \partial_t u = \Delta u +(\sigma-\mu u)\,u$$
and to the stationary case of equilibrium solution described
by the elliptic equation
$$ \Delta u +(\sigma-\mu u)\,u = 0.$$
\medskip

In this paper we will consider two 
variants of the
latter equation, motivated by the 
nonlocal features of the
population. \medskip

First of all, the diffusion operator of the population
is considered to be nonlocal, that is, we replace the
Gaussian diffusion by the one induced by L\'evy flights.
These types of nonlocal dispersal strategy have been observed
in nature and may be related to optimal hunting strategies
and adaptation to the environment stimulated by the natural
selection, see e.g.~\cite{Vis_al, Hum_al} for experimental results
and~\cite{VEN} for divulgative explanations of these phenomena in popular
magazines. {F}rom the mathematical point of view, taking into
account this kind of nonlocal diffusion translates in our setting
into the analysis of logistic equations driven by fractional Laplace operators.\medskip

Moreover, we take into account the possibility that also the
increasing rate of the species has a nonlocal character. This
feature is motivated in concrete cases by the fact that a population
takes advantage not only of the resources that are exactly in the
area in which they permanent settle, but also of the ones that
are ``at their reach''
(say, a ``giraffe's neck'' effect).
This nonlocal feature will be modeled for us by the convolution
with an integrable kernel (from the mathematical point of view,
we remark that the two types of nonlocal operators considered
are very different, since the fractional Laplacian causes
a loss of differentiability on the function, while
the convolution produces a regularizing effect).\medskip

The precise mathematical formulation that we consider is the following.
Given~$s\in(0,1)$, we consider the fractional Laplacian
\begin{equation}\label{NOA}
(-\Delta)^s u(x):= 2s\,(1-s)\,
PV\,\int_{\R^n}\frac{u(x)-u(y)}{|x-y|^{n+2s}}\,dy.\end{equation}
The notation ``$PV$'' denotes, as customary, the singular integral
taken in the ``principal value'' sense, that is
$$ PV\,\int_{\R^n}\frac{u(x)-u(y)}{|x-y|^{n+2s}}\,dy:=
\lim_{\delta\to0} \int_{\R^n\setminus B_\delta(x)}
\frac{u(x)-u(y)}{|x-y|^{n+2s}}\,dy.$$
The constant~$s\,(1-s)$ in~\eqref{NOA}
is just a normalizing factor, to allow ourselves
to consider the case~$s=1$ as a limit.
Indeed, with this choice,
$$ \lim_{s\to1}(-\Delta)^su(x) =c_\star \sum_{i=1}^n
\frac{\partial^2 u}{\partial x_i^s} (x)=:-\Delta u(x),$$
for a suitable normalizing constant~$c_\star>0$,
only depending on~$n$,
for any~$u\in C^2(\R^n)\cap L^\infty(\R^n)$.

The stationary logistic equation that we study is then
$$ -(-\Delta)^s u +(\sigma-\mu u)\,u +\tau(J*u)= 0,$$
where~$\sigma$, $\mu$ and~$J$ are nonnegative functions, $\tau\ge0$ is a constant and~$s\in(0,1]$. 
As usual, $J*u$ denotes the convolution between two functions, that is, for any~$x\in\R^n$,
$$ (J*u)(x):=\int_{\R^n} J(x-y)u(y)\, dy.$$
We also assume that the convolution kernel is even and normalized
with total mass~$1$, that is
\begin{equation}\label{unoint}
\int_{\R^n}J(x)\,dx=1 \end{equation}
and
\begin{equation}\label{symm0990}
J(-x)=J(x) \quad {\mbox{ for any~$x\in\R^n$.}} \end{equation}
We consider two types of setting for our equation:
the bounded domain with Dirichlet datum (corresponding
to a confined
environment with hostile surrounding areas) and the periodic case.
These two cases will be discussed in detail in the forthcoming subsections.
\medskip

For recent investigations
of different nonlocal equations arising
in biological contexts, see
e.g.~\cite{MR2601079, MR3082317, MR3103010, MR3274582, 2015arXiv150301629M}
and the references therein.

\subsection{Bounded domains with Dirichlet data}

The environment with hostile borders is modeled in our case
by the following equation:
\begin{equation}\label{EQ}
\left\{ \begin{matrix}(-\Delta)^{s} u = 
(\sigma - \mu u)\,u +\tau (J*u) &{\mbox{ in }}\Omega,
\\u=0 & {\mbox{ outside }}\Omega,\\
u\ge0&{\mbox{ in }}\R^n.\end{matrix}\right.
\end{equation}
We will present an existence theory for nontrivial solutions
and we will compare local and nonlocal behaviors of the population,
analyzing their effectiveness in terms of the resource and of the domain.

In further detail, 
we consider the (possibly fractional) critical Sobolev 
exponent~$2^*_s:=2n/(n-2s)$ and we state
a general existence result as follows:

\begin{theorem}\label{EX56AJJ}
Let~$\Omega$ be a bounded Lipschitz domain.
Assume that~$\sigma\in L^m(\Omega)$,
for some~$m\in (2^*_s/(2^*_s-2),+\infty]$,
and that~$(\sigma+\tau)^3\mu^{-2}\in L^1(\Omega)$. 
Then, there exists a solution of~\eqref{EQ}.
\end{theorem}

To study the solutions obtained by Theorem~\ref{EX56AJJ}
it is useful to compare them to the domain using a spectral analysis.
For this, we denote by~$\lambda_s(\Omega)$
the first Dirichlet eigenvalue for~$(-\Delta)^s$ in~$\Omega$, i.e.
$$ \lambda_s(\Omega):= \inf
s\,(1-s)\iint_{Q_\Omega}\frac{|u(x)-u(y)|^2}{|x-y|^{n+2s}}\,dx\,dy,$$
where
\begin{equation}\label{4bis}
Q_\Omega:= (\Omega\times\R^n)\cup((\R^n\setminus\Omega)\times\Omega)
\end{equation} 
and the infimum is taken under the conditions
that~$\| u\|_{L^2(\R^n)}=1$
and~$u=0$ outside~$\Omega$, 
if~$s\in(0,1)$, and, as classical,
$$ \lambda_1(\Omega):= c_\star \inf_{{\| u\|_{L^2(\R^n)}=1}\atop{u\in H^1_0(\Omega)}}
\int_\Omega|\nabla u|^2\,dx.$$
For a detailed study of these eigenvalues
(also in the nonlocal case) see for instance
Appendix~A in~\cite{MR3002745}.

The existence of nontrivial solutions to~\eqref{EQ} can be 
characterized in terms of these first eigenvalues:
roughly speaking, when the resource~$\sigma$ is too small,
the only solution of~\eqref{EQ} is the one identically zero,
i.e. all the population dies; viceversa, if the resource~$\sigma$
is large enough, there exists a positive solution. 

More precisely, we have the following:

\begin{theorem}\label{COP}
Let~$\Omega$ be a bounded Lipschitz domain.
Assume that~$\sigma\in L^m(\Omega)$,
for some~$m\in (2^*_s/(2^*_s-2),+\infty]$,
and that~$(\sigma+\tau)^3\mu^{-2}\in L^1(\Omega)$.
Then:
\begin{itemize}
\item if $\sup_{\Omega}\sigma+\tau\le\lambda_s(\Omega)$ 
then the only solution of~\eqref{EQ}
is the one identically zero;
\item if
$\inf_{\Omega}\sigma\ge\lambda_s(\Omega)$ with strict inequality on a set of positive
measure
and~$\mu\in L^1(\Omega)$,
then~\eqref{EQ}
possesses a solution~$u$ such that~$u>0$ in~$\Omega$.
\end{itemize}
\end{theorem}

A consequence of Theorem~\ref{COP} is that
nonlocal species can better adapt to sparse resources.
For instance, there exist examples of disjoint domains~$\Omega_1$
and~$\Omega_2$ such that the resource in each single~$\Omega_i$
is not sufficient for the species to survive, but the combined
resources in the union of the domains can be used by a nonlocal
population efficiently enough. A formal statement goes as follows:

\begin{theorem}\label{CONG}
Let~$s\in(0,1)$.
Let~$\Omega_1$ be a domain in~$\R^n$, and~$\Omega_2$ be
a domain congruent to~$\Omega_1$, with~$\overline{\Omega_1}\cap\overline{\Omega_2}=\varnothing$.
Then, there exists~$\sigma\in(0,+\infty)$ such that
the only solution of
$$ \left\{ \begin{matrix}(-\Delta)^{s} u =
(\sigma - \mu u)\,u &{\mbox{ in }}\Omega_i,
\\u=0 & {\mbox{ outside }}\Omega_i,\\
u\ge0&{\mbox{ in }}\R^n\end{matrix}\right. $$
is the trivial one, for any~$i\in\{1,2\}$,
but the equation
$$ \left\{ \begin{matrix}(-\Delta)^{s} u =
(\sigma - \mu u)\,u &{\mbox{ in }}\Omega_1\cup\Omega_2,
\\u=0 & {\mbox{ outside }}\Omega_1\cup\Omega_2,\\
u\ge0&{\mbox{ in }}\R^n\end{matrix}\right. $$
admits a positive solution in~$\Omega_1\cup\Omega_2$.
\end{theorem}

Also, in light of Theorem~\ref{COP}
it is interesting to determine for which~$s$
positive solutions of~\eqref{EQ} may occur. 
Roughly speaking,
when~$\Omega$ is ``small'', 
the strongly diffusive species corresponding to
small values of~$s$ may be favored.
Viceversa, when~$\Omega$ is ``large'', the
species corresponding to small~$s$ may be annihilated.
As a prototype example we present the following two results:

\begin{proposition}\label{NONAM}
Let~$\Omega$ be a bounded Lipschitz domain and set
$$ \Omega_r := \{ rx, \; x\in\Omega\}.$$ 

Then the equation
$$ \left\{ \begin{matrix}(-\Delta)^{s} u = 
(1- u)\,u +\tau\,(J*u) &{\mbox{ in }}\Omega_r,
\\u=0 & {\mbox{ outside }}\Omega_r,\\
u\ge0&{\mbox{ in }}\R^n\end{matrix}\right.
$$
admits a nontrivial solution if and only if
$$ r > \left( \lambda_s(\Omega) \right)^{\frac{1}{2s}}.$$
\end{proposition}

\begin{theorem}\label{EXT}
Fix~$s$, $S\in(0,1]$, with~$s<S$.
Let~$\Omega$ be a bounded Lipschitz domain and set
$$ \Omega_r := \{ rx, \; x\in\Omega\}.$$
Let also $J$ be a nonnegative function satisfying \eqref{unoint} and \eqref{symm0990}.

Then there exist~$\overline r>\underline r>0$
such that
\begin{itemize}
\item if~$r\in(0,\underline r)$, then there exist~$\sigma_r$, $\tau_r\in(0,+\infty)$ such that
the equation
\begin{equation}\label{ER2} \left\{ \begin{matrix}(-\Delta)^{s} u =
(\sigma_r- u)\,u +\tau_r\,(J*u) &{\mbox{ in }}\Omega_r,
\\u=0 & {\mbox{ outside }}\Omega_r,\\
u\ge0&{\mbox{ in }}\R^n\end{matrix}\right.
\end{equation}
admits a nontrivial solution while the equation
\begin{equation}\label{ER1} \left\{ \begin{matrix}(-\Delta)^{S} u =
(\sigma_r- u)\,u +\tau_r\,(J*u) &{\mbox{ in }}\Omega_r,
\\u=0 & {\mbox{ outside }}\Omega_r,\\
u\ge0&{\mbox{ in }}\R^n\end{matrix}\right.
\end{equation}
admits only the trivial solution;
\item viceversa, if~$r\in(\overline r,+\infty)$ then 
there exist~$\sigma_r$, $\tau_r\in(0,+\infty)$ such that equation~\eqref{ER2}
only admits the trivial solution, while equation~\eqref{ER1}
admits a nontrivial solution.\end{itemize}
\end{theorem}

The biological interpretation of Theorem~\ref{EXT}
is that ``large'' environments are ``more favorable''
to ``local'' populations (namely, the population with
faster diffusion related to~$(-\Delta)^s$ is extinguished,
while the population with
slower diffusion related to~$(-\Delta)^S$ is still alive); viceversa,
``small'' environments are ``more favorable''
to ``nonlocal'' populations (namely, in this case it is
the population with
slower diffusion~$(-\Delta)^S$ that is extinguished,
while the population with
faster diffusion~$(-\Delta)^s$ persists).\medskip

Another relevant question in this framework
is whether or not the population
fits the resources. An easy observation is that, if $\tau=0$, 
the population never overcomes the maximal available resource. 
This follows from the more general result:

\begin{lemma}\label{easy}
If~$\sigma\in L^\infty(\Omega)$ and~$u$ is a solution of
\begin{equation*}
\left\{ \begin{matrix}(-\Delta)^{s} u =
(\sigma- u)\,u +\tau (J*u) &{\mbox{ in }}\Omega,
\\u=0 & {\mbox{ outside }}\Omega,\\
u\ge0&{\mbox{ in }}\R^n, \end{matrix}\right.
\end{equation*}
then~$u\le \|\sigma\|_{L^\infty(\Omega)}+\tau$.
\end{lemma}

It is conceivable to think that large resources in a given region
favor, at least locally, large density populations.
We show indeed that there is a linear dependence on the largeness
of the resource and the population density (independently
on how large the resource is), according to the following result:

\begin{theorem}\label{abundance}
Let~$R>r>0$.
Let~$\Omega$ be a bounded Lipschitz domain, with~$\overline{B_{R}}\subset\Omega$.
Then, there exist~$c_o\in(0,1)$ only depending on~$n$, $s$,
$R$ and~$r$, and~$M_o>0$ only 
depending on~$n$, $s$ and~$R$,
such that if~$M\ge M_o$ and $\sigma\ge M$
in~$B_R$,
then there exists a solution~$u$ of
\begin{equation*}
\left\{ \begin{matrix}(-\Delta)^{s} u =
(\sigma- u)\,u +\tau\,(J*u) &{\mbox{ in }}\Omega,
\\u=0 & {\mbox{ outside }}\Omega,\\
u\ge0&{\mbox{ in }}\R^n,\end{matrix}\right.
\end{equation*}
such that~$u\ge c_oM$ in~$B_r$.
\end{theorem}

Next result stresses the fact that nonlocal populations
can efficiently plan their distribution in order to consume
and possibly beat the given resources in a given ``strategic region''
(up to a small error). That is, fixing a region of interest, say the ball~$B_1$,
one can find a solution of a (slightly perturbed by an error~$\varepsilon$) logistic equation
in~$B_1$ which exhausts the resources in~$B_1$ and which
vanishes outside~$B_{R_\varepsilon}$, for some (possibly large)~$R_\varepsilon>1$.
The ``strategic plan'' in this framework
consists in the fact that, in order for
the population to consume all the given resource in~$B_1$,
the distribution in~$B_{R_\varepsilon}\setminus B_1$ must be appropriately
adjusted (in particular, the logistic equation is
not satisfied in~$B_{R_\varepsilon}\setminus B_1$,
where the population needs to be ``artificially'' settled from outside).
The detailed statement of such result goes as follows:

\begin{theorem}\label{BOVI}
Let~$s\in(0,1)$ and~$k\in\N$, with~$k\ge2$.
Assume that
$$ \inf_{\overline{B_2}}\mu >0,\qquad\;\inf_{\overline{B_2}}\sigma>0,$$
and that
$\sigma$, $\mu\in C^k(\overline{B_2})$. Fix~$\varepsilon\in(0,1)$.
Then, there exist a nonnegative function~$u_\varepsilon$, $R_\varepsilon>2$
and~$\sigma_\varepsilon\in C^k(\overline{B_1})$ such that
\begin{eqnarray}
&& \label{LAJ6:OV:1}
(-\Delta)^s u_\varepsilon = (\sigma_\varepsilon -\mu\, u_\varepsilon)\,
u_\varepsilon +\tau (J*u_\varepsilon) \qquad{\mbox{ in $B_1$,}}\\
&& \label{LAJ6:OV:2}
u_\varepsilon =0 \qquad{\mbox{ in $\R^n\setminus B_{R_\varepsilon}$}},\\
&& \label{LAJ6:OV:CC}
\|\sigma_\varepsilon-\sigma\|_{C^k(\overline{B_1})}
\le\varepsilon\\
{\mbox{and }} \label{LAJ6:OV:3}
&& u_\varepsilon\ge \mu^{-1}\sigma_\varepsilon  \qquad{\mbox{ in $B_1$.}}
\end{eqnarray}
\end{theorem}

In light of Lemma~\ref{easy}
and Theorems~\ref{abundance} and~\ref{BOVI},
a relevant question is also whether or not the population can
beat the resource, i.e. whether or not the set~$\{u>\sigma\}$
is void. Notice indeed that Lemma~\ref{easy} says that, if $\tau=0$, 
this does not occur for constant resources~$\sigma$. Nevertheless, 
when the resource is oscillatory, then this phenomenon
occurs, thanks to the diffusive terms which allow the species
to somewhat attains resources ``from somewhere else''.
Namely we have the following result:

\begin{theorem}\label{BEAT0}
Let~$R>r>0$ and~$\Omega$ be a bounded Lipschitz domain satisfying
the exterior ball condition and such that~$\overline{B_R}\subset \Omega$.
Let~$M_o$ be as in Theorem~\ref{abundance}.

Let~$\sigma_0\in C(\overline\Omega)$ be such that~$\sigma_0\ge M_o$
in~$B_R$. Assume also that there exists~$x_0\in\Omega$
such that~$\sigma_0(x_0)=0$, and, 
for any~$m\in[0,1]$, set~$\sigma_m:=\sigma+m$.
Then there exists~$m_0>0$ such that for any~$m\in(0,m_0)$
there exists a solution of
\begin{equation}\label{df}\left\{ \begin{matrix}(-\Delta)^{s} u =
(\sigma_m- u)\,u &{\mbox{ in }}\Omega,
\\u=0 & {\mbox{ outside }}\Omega,\\
u\ge0&{\mbox{ in }}\R^n\end{matrix}\right.
\end{equation}
for which~$\{u>\sigma_m\}$
is nonvoid. 
\end{theorem}

\subsection{Periodic environments}

We now turn our attention to a periodic environment,
i.e. we suppose that~$\sigma$ and~$\mu$ are periodic with respect
to translations in~$\Z^n$ and we look for periodic solutions.
In this framework, the equation that we take into account is
\begin{equation}\label{EQ:PER}
\left\{ \begin{matrix}(-\Delta)^{s} u =
(\sigma - \mu u)\,u +\tau (J*u) &{\mbox{ in }}\R^n,
\\u(x+k)=u(x) & {\mbox{ for any }}k\in\Z^n,\\
u>0&{\mbox{ in }}\R^n,\end{matrix}\right.
\end{equation}
We suppose here that~$\sigma$ and~$\mu$ are bounded
and periodic functions
(with respect to the lattice~$\Z^n$), that~$\mu$ is
positive and bounded away
from zero and that~$J$ is compactly supported.

In this setting, we obtain the following existence result for periodic
solutions:

\begin{theorem}\label{0o56:TH}
Assume that
\begin{equation}\label{KJA:AKK}
{\mbox{either~$\sigma$ is not identically zero or~$\tau>0$.}}
\end{equation}
Then, there exists a solution of~\eqref{EQ:PER}.
\end{theorem}

We remark that the solutions obtained in Theorem~\ref{0o56:TH}
are in general not constant (for instance, when~$\mu$
is constant and~$\sigma$ is not). But when both~$\sigma$
and~$\mu$ are constant then the periodic solutions
need also to be constant, according to the following result:

\begin{theorem}\label{0o56:TH:2}
Let~$u$ be a positive solution of~$(-\Delta)^s u=
(\sigma-\mu u)u+\tau (J*u)$ in~$\R^n$. Assume that~$u$ is periodic
with respect to~$\Z^n$ and that~$\sigma\in(0,+\infty)$, $\mu\in(0,+\infty)$
and~$\tau\in[0,+\infty)$ are all constant.

Then, $u$ is also constant, and constantly equal to~$(\sigma+\tau)/\mu$.
\end{theorem}

\subsection{A transmission problem}

Now, inspired by the recent work in~\cite{KRI},
we consider a transmission model in which the population
is made of two species
(or of one population that adapts to two different
environments), one with a local behavior
in a domain~$\Omega_1$, and one with a nonlocal behavior
in a domain~$\Omega_2$, with~$\Omega_1\cap\Omega_2=\varnothing$.
The transmission problem occurs between~$\Omega_i$ and its complement,
for~$i\in\{1,2\}$, and it is modeled by positive parameters~$\nu_i$.

More precisely, we take two disjoint,  bounded and
Lipschitz domain~$\Omega_1$ and~$\Omega_2\subset\R^n$.
We define~$\Omega:=\Omega_1\cup\Omega_2$ and
\begin{equation}\label{TP}\begin{split}
{\mathcal{T}}(u)\;
&:=\;\frac12 \int_{\Omega_1}|\nabla u|^2\,dx
+\frac{s\,(1-s)}{2}\iint_{\Omega_2\times\Omega_2}
\frac{|u(x)-u(y)|^2}{|x-y|^{n+2s}}\,dx\,dy\\
&\qquad+\sum_{i=1}^2 
\frac{\nu_i\,s_i\,(1-s_i)}{2}\iint_{\Omega_i\times(\R^n\setminus\Omega_i)}
\frac{|u(x)-u(y)|^2}{|x-y|^{n+2s_i}}\,dx\,dy
+\int_\Omega \frac{\mu\,|u|^3}{3} - \frac{\sigma \,u^2}{2}\,dx.
\end{split}\end{equation}
Here, $s$, $s_1$, $s_2\in(0,1)$, $\sigma$, 
$\mu\in L^\infty(\Omega,[0,+\infty))$
with~$\mu\ge \mu_o$, for some~$\mu_o>0$.

In this setting, we have the following existence result:

\begin{theorem}\label{TR:MOD:MI}
The functional~${\mathcal{T}}$ attains its minimum among the functions~$u\in
L^2(\Omega)$ for which
\begin{eqnarray*}
&&\frac12 \int_{\Omega_1}|\nabla u|^2\,dx
+\frac{s\,(1-s)}{2}\iint_{\Omega_2\times\Omega_2}
\frac{|u(x)-u(y)|^2}{|x-y|^{n+2s}}\,dx\,dy\\
&&\qquad+\sum_{i=1}^2
\frac{\nu_i\,s_i\,(1-s_i)}{2}\iint_{\Omega_i\times(\R^n\setminus\Omega_i)}
\frac{|u(x)-u(y)|^2}{|x-y|^{n+2s_i}}\,dx\,dy \;<\;+\infty
,\end{eqnarray*}
and such that~$u=0$ a.e. outside~$\Omega$.

Also, such minimizer is nonnegative.
\end{theorem}

It is worth to point out that minimizers of~${\mathcal{T}}$ satisfy
the equations
\begin{equation}\label{EQ:TP}
\begin{split}
&\quad -\Delta u +
\frac{\nu_1\,s_1\,(1-s_1)}{2}\int_{\R^n\setminus \Omega_1}
\frac{u(x)-u(y)}{|x-y|^{n+2s_1}}\,dy\\ &\qquad\qquad\qquad+
\frac{\nu_2\,s_1\,(1-s_2)}{2}\int_{\Omega_2}
\frac{u(x)-u(y)}{|x-y|^{n+2s_2}}\,dy = (\sigma-\mu u) \,u
\quad{\mbox{ in }}\Omega_1\\
{\mbox{and }}&\quad 
2s\,(1-s)\,
PV\,\int_{\Omega_2}\frac{u(x)-u(y)}{|x-y|^{n+2s}}\,dy
+\frac{\nu_1\,s_1\,(1-s_1)}{2}\int_{\Omega_1}
\frac{u(x)-u(y)}{|x-y|^{n+2s_1}}\,dy\\ &\qquad\qquad\qquad+
\frac{\nu_2\,s_1\,(1-s_2)}{2}\int_{\R^n\setminus\Omega_2}
\frac{u(x)-u(y)}{|x-y|^{n+2s_2}}\,dy = (\sigma-\mu u) \,u
\quad{\mbox{ in }}\Omega_2,
\end{split}
\end{equation}
in the weak sense (and also pointwise, by Theorem~5.5(3)
in~\cite{KRI} and Theorem~1 in~\cite{weak}).

The biological interpretation of equation~\eqref{EQ:TP}
is that the population has local behavior in~$\Omega_1$,
with nonlocal interactions outside~$\Omega_1$, and a nonlocal
transmission between the domains~$\Omega_1$ and~$\Omega_2$
takes place. See also~\cite{KRI}
for additional comments and motivations.

The existence/nonexistence of nontrivial solutions
in dependence of the spectral analysis of the domain
will be addressed in the following result. To this end,
we define~$\lambda_\star(\Omega)$
the first Dirichlet eigenvalue for
the operator in~\eqref{TP}. Namely, we set
\begin{equation}\label{89RFUSjKaA}
\lambda_\star(\Omega)\;:=\; \inf {\mathcal{T}}_o(u),\end{equation}
where
\begin{eqnarray*}
{\mathcal{T}}_o(u)&:=&
\int_{\Omega_1}|\nabla u|^2\,dx
+s\,(1-s)\iint_{\Omega_2\times\Omega_2}
\frac{|u(x)-u(y)|^2}{|x-y|^{n+2s}}\,dx\,dy
\\ &&\qquad+\sum_{i=1}^2
\nu_i\,s_i\,(1-s_i)\iint_{\Omega_i\times(\R^n\setminus\Omega_i)}
\frac{|u(x)-u(y)|^2}{|x-y|^{n+2s_i}}\,dx\,dy,\end{eqnarray*}
and the infimum in~\eqref{89RFUSjKaA}
is taken under the conditions
that~$\| u\|_{L^2(\R^n)}=1$ and~$u=0$ a.e. outside~$\Omega$.
In this setting, we obtain a result similar to Theorem~\ref{COP}
for the transmission problem in~\eqref{TP}:

\begin{theorem}\label{COP:TP}
In the setting above,
\begin{itemize}
\item if $\sup_{\Omega}\sigma\le\lambda_\star(\Omega)$
then the only solution of~\eqref{EQ:TP}
is the one identically zero;
\item if
$\inf_{\Omega}\sigma\ge\lambda_\star(\Omega)$ with strict 
inequality on a set of positive
measure then~\eqref{EQ:TP}
possesses a solution~$u$ such that~$u>0$ in~$\Omega_1\cup\Omega_2$.
\end{itemize}
\end{theorem}

\subsection{Organization of the paper}

The rest of the paper is organized as follows:
in Section~\ref{SE:2} we discuss the existence of a solution
by energy minimization and we prove Theorem~\ref{EX56AJJ}.

Then, in Section~\ref{SA:3}, we discuss the
qualitative properties of the solution
and we present a proof of Theorem~\ref{COP}.

In Sections~\ref{SA:4}, \ref{SA:5} and~\ref{SA:6}
we discuss how the population adapts to the resources
and we give the proof of
Theorem~\ref{CONG}, Proposition~\ref{NONAM},
Theorem~\ref{EXT}, Lemma~\ref{easy} and Theorem~\ref{abundance}.

The strongly nonlocal diffusive strategy is considered
in Section~\ref{SA:6:BIS},
where we prove Theorem~\ref{BOVI}.

The case in which the population actually beats the resource is
discussed in Section~\ref{SA:7}, where Theorem~\ref{BEAT0} is proved.

The existence/nonexistence of nontrivial periodic solutions
in a periodic environment is taken into account
in Section~\ref{SA:PP} with the
proofs of Theorems~\ref{0o56:TH} and~\ref{0o56:TH:2}.

Then, in Section~\ref{TRA:SEC}, we consider
the transmission problem and we prove Theorems~\ref{TR:MOD:MI}
and~\ref{COP:TP}.

\section{Existence theory and proof of Theorem~\ref{EX56AJJ}}\label{SE:2}

The proof of Theorem~\ref{EX56AJJ} is based on a minimization
argument.
More precisely, in order to deal with problem~\eqref{EQ}, 
if~$s\in(0,1)$, given~$u\in L^1_{\rm loc}(\R^n)$ with~$u=0$
a.e. outside~$\Omega$,
we consider the energy functional
$$ {\mathcal{E}}(u):=\frac{s\,(1-s)}{2} \iint_{Q_\Omega}\frac{|u(x)-u(y)|^2}{
|x-y|^{n+2s}}\,dx\,dy
+\int_\Omega \frac{\mu\,|u|^3}{3}
-\frac{\sigma\,u^2}{2} - \frac{\tau\,u\,(J*u)}{2} \,dx,$$
where~$Q_\Omega$ is defined in~\eqref{4bis}. 

When~$s=1$, instead we consider the standard energy functional
$$ {\mathcal{E}}(u):=\frac{c_\star}{2} \int_\Omega\frac{|\nabla u|^2}{2}
+\frac{\mu\,|u|^3}{3}
-\frac{\sigma\,u^2}{2}- \frac{\tau\,u\,(J*u)}{2}\,dx,$$
with condition~$u\in H^1_0(\Omega)$.

It is worth to point out that solutions of~\eqref{EQ}
are strictly positive, unless they vanish identically:

\begin{lemma}\label{positive}
Let~$u$ be
a nonnegative solution of~$
(-\Delta)^{s} u =
(\sigma - \mu u)\,u +\tau (J*u)$ in~$\Omega$.
Then either~$u>0$ in~$\Omega$ or it vanishes identically.
\end{lemma}

\begin{proof}
Suppose that~$u(z)=0$
for some~$z\in\Omega$ and, by contradiction, that~$u>0$
in a set of positive measure.
Then~$u(z+x)-u(z)=u(z+x)\ge0$ for any~$x\in\R^n$,
and in fact strictly positive in a set of positive measure.
Accordingly, $(-\Delta)^s u(z)<0$. Nevertheless, from~\eqref{EQ},
we have that
$$ (-\Delta)^{s} u (z)= (\sigma(z) - \mu(z) u(z))\,u(z) +\tau (J*u)(z)=\tau (J*u)(z)\ge 0,$$
which is a contradiction.
\end{proof}

Equation~\eqref{EQ} has a variational structure, according to the following
observation:

\begin{lemma}\label{90}
The Euler-Lagrange 
equation associated to the energy functional~${\mathcal{E}}$ at
a nonnegative function~$u$ is~\eqref{EQ}.
\end{lemma}

\begin{proof} We denote by
$$\mathcal{J}(u):=\int_{\Omega}\frac{\tau\,u\,(J*u)}{2}\,dx.$$ 
If~$\phi\in C^\infty_0(\Omega)$ and~$\epsilon\in(-1,1)$, we have that 
\begin{eqnarray*}
&&\mathcal{J}(u+\epsilon\phi)\\&=&\frac{\tau}{2}\int_{\Omega}(u+\epsilon\phi)(x) \big( J*(u+\epsilon\phi)\big)(x)\,dx
\\&=& \frac{\tau}{2}\int_{\Omega}u(x)(J*u)(x) +\epsilon \Big[ (u(x)(J*\phi)(x)+\phi(x)(J*u)(x)\Big] 
+ \epsilon^2 \phi(x)(J*\phi)(x)\,dx.
\end{eqnarray*}
As a consequence, 
\begin{equation}\label{kkgjtrekhjtr}
\frac{d\mathcal{J}}{d\epsilon}(u+\epsilon\phi)\Big|_{\epsilon=0} = \frac{\tau}{2}\int_{\Omega} 
\left(u(x)(J*\phi)(x)+\phi(x)(J*u)(x)\right)\, dx.
\end{equation}
Now we recall that $u$ and $\phi$ vanish outside~$\Omega$ and we use~\eqref{symm0990} to see that 
\begin{eqnarray*}
&& \int_{\Omega} u(x)(J*\phi)(x)\,dx = \int_{\R^n} u(x)\left( \int_{\R^n}J(x-y)\phi(y)\,dy\right)\,dx \\
&&\qquad = \int_{\R^n}\phi(y) \left(\int_{\R^n}J(x-y)u(x)\,dx\right)\,dy 
= \int_{\R^n}\phi(y) \left(\int_{\R^n}J(y-x)u(x)\,dx\right)\,dy \\
&&\qquad = \int_{\R^n}\phi(y) (J*u)(y)\,dy
=  \int_{\Omega}\phi(y) (J*u)(y)\,dy. 
\end{eqnarray*}
Using this into~\eqref{kkgjtrekhjtr} we obtain that 
$$ \frac{d\mathcal{J}}{d\epsilon}(u+\epsilon\phi)\Big|_{\epsilon=0} = \tau\int_{\Omega}\phi(x)(J*u)(x)\,dx.$$ 

With this, the case~$s=1$ is standard, so we consider the case~$s\in (0,1)$.
If~$\phi\in C^\infty_0(\Omega)$,
we have
$$ \iint_{Q_\Omega}\frac{\big(u(x)-u(y)\big)\big(\phi(x)-\phi(y)\big)}{
|x-y|^{n+2s}}\,dx\,dy =
\iint_{\R^{2n}}\frac{\big(u(x)-u(y)\big)\big(\phi(x)-\phi(y)\big)}{
|x-y|^{n+2s}}\,dx\,dy,$$
which gives the desired result.
\end{proof}

In the light of Lemma~\ref{90},
to prove existence of solutions, it is useful to look at the minimizing
problem for~${\mathcal{E}}$. 
We first show the following useful inequality: 

\begin{lemma}\label{young}
Let $v$, $w\in L^2(\Omega)$ with~$v=0=w$ a.e. outside~$\Omega$. Then 
\begin{equation}\label{usodopo}
\int_{\Omega}v(x)(J*w)(x)\,dx \le \|v\|_{L^2(\Omega)}\|w\|_{L^2(\Omega)}.
\end{equation}
\end{lemma}

\begin{proof}
By the H\"older Inequality with exponents equal to~2 and the Young Inequality 
for convolutions with exponents~1 and~2, we have that
\begin{eqnarray*}&&
\int_{\Omega}v(x)(J*w)(x)\,dx \le \|v\|_{L^2(\Omega)} \|J*w\|_{L^2(\R^n)} 
\\&&\qquad\qquad\le \|v\|_{L^2(\Omega)} \|J\|_{L^1(\R^n)}\|w\|_{L^2(\Omega)} = \|v\|_{L^2(\Omega)}\|w\|_{L^2(\Omega)}, 
\end{eqnarray*}
where~\eqref{unoint} was also used. This shows \eqref{usodopo}. 
\end{proof}

Then the following existence result holds:

\begin{proposition}\label{existence}
Let~$\Omega$ be a bounded Lipschitz domain.

Assume that~$\sigma\in L^m(\Omega)$,
for some~$m\in (2^*_s/(2^*_s-2),+\infty]$,
and that~$(\sigma+\tau)^3\mu^{-2}\in L^1(\Omega)$. 
Let also
$$ p:=\frac{2}{1-\frac{1}{m}}.$$
Then~${\mathcal{E}}$ attains its minimum among the functions~$u\in L^p(\Omega)$
for which
$$ \iint_{Q_\Omega}\frac{|u(x)-u(y)|^2}{
|x-y|^{n+2s}}\,dx\,dy<+\infty$$
and such that~$u=0$ a.e. outside~$\Omega$.

Moreover, there exists a nonnegative minimizer.
Finally, if~$u$ is such minimizer, it is a solution of~\eqref{EQ}.
\end{proposition}

\begin{proof} We deal with the case~$s\in(0,1)$, since the case~$s=1$
is similar, and simpler. The proof is by direct methods. 
First, we 
notice that~$p\in [2,2^*_s)$ and
\begin{equation}\label{DP}
\frac{2}{p}+\frac{1}{m}=1.
\end{equation}
By \eqref{usodopo} (used here with $v:=u$ and $w:=u$) we have that
\begin{equation}\label{lakdetpieryr}
\int_{\Omega}\frac{\tau\,u\,(J*u)}{2}\,dx \le \frac{\tau}{2}\int_{\Omega}|u|^2\,dx.
\end{equation}
Furthermore, we use the Young Inequality, with exponents~$3/2$ and~$3$,
to see that
\begin{equation}\label{7bis}
\frac{(\sigma+\tau)\,u^2}{2} =
\frac{\mu^{2/3}|u|^2}{2^{2/3}}
\cdot \frac{\sigma+\tau}{2^{1/3} \mu^{2/3}}\le
\frac{\mu\,|u|^3}{3}+\frac{(\sigma+\tau)^3}{6 \mu^2}.\end{equation}
As a consequence of this and~\eqref{lakdetpieryr},
$$ \int_{\Omega}\frac{\mu\,|u|^3}{3}-\frac{\sigma\,u^2}{2}-\frac{\tau\,u\,(J*u)}{2}\,dx 
\ge - \int_{\Omega} \frac{(\sigma+\tau)^3}{6\mu^2}\,dx.$$ 
This implies that
$$ {\mathcal{E}}(u)\ge\frac{s\,(1-s)}{2} \iint_{Q_\Omega}\frac{|u(x)-u(y)|^2}{
|x-y|^{n+2s}}\,dx\,dy-\kappa,$$
for~$\kappa:=\|(\sigma+\tau)^3\mu^{-2}\|_{L^1(\Omega)}/6$. So we can take
a minimizing sequence~$u_j$. We may suppose that
$$ 0={\mathcal{E}}(0)\ge{\mathcal{E}}(u_j)\ge 
\frac{s\,(1-s)}{2} \iint_{Q_\Omega}\frac{|u_j(x)-u_j(y)|^2}{
|x-y|^{n+2s}}\,dx\,dy-\kappa.$$
So we set
$$ \| u_j\|:=\sqrt{ s\,(1-s)
\iint_{Q_\Omega}\frac{|u_j(x)-u_j(y)|^2}{
|x-y|^{n+2s}}\,dx\,dy }.$$
We obtain that
$$ \sqrt{s\,(1-s)
\iint_{\R^{2n}}\frac{|u_j(x)-u_j(y)|^2}{
|x-y|^{n+2s}}\,dx\,dy}= \| u_j\| \le \sqrt{2\kappa}.$$
Hence, by compactness, up to a subsequence~$u_j$ converges to some~$u$
in~$L^p(\Omega)$ and a.e. in~$\R^n$. So we recall~\eqref{DP}
and we find that
\begin{eqnarray*}&&
\limsup_{j\to+\infty} \int_\Omega \sigma(u_j^2-u^2)\,dx=
\limsup_{j\to+\infty} \int_\Omega \sigma(u_j+u)(u_j-u)\,dx\\ &&\qquad\qquad\le
\limsup_{j\to+\infty} \|\sigma\|_{L^m(\Omega)}
\|u_j+u\|_{L^p(\Omega)}
\|u_j-u\|_{L^p(\Omega)}=0.\end{eqnarray*}
Furthermore, 
\begin{equation}\label{asrewrqrwqerq}
\int_{\Omega}\big( u_j\,(J*u_j)-u\,(J*u) \big)\,dx = \int_{\Omega} (u_j-u)\,(J*u_j)\,dx +\int_{\Omega}(J*u_j-J*u)\,u\,dx.
\end{equation}
Now, by~\eqref{usodopo} with~$v:=u_j-u$ and~$w:=u_j$ we obtain 
\begin{equation}\label{8bis}
\limsup_{j\to+\infty} \int_{\Omega} (u_j-u)\,(J*u_j)\,dx \le 
\limsup_{j\to+\infty} \|u_j-u\|_{L^2(\Omega)}\|u_j\|_{L^2(\Omega)}=0.\end{equation}
Moreover, making again use of~\eqref{usodopo} with~$v:=u$ and~$w:=u_j-u$, we have that 
\begin{equation}\label{8ter} \begin{split}&
\limsup_{j\to+\infty} \int_{\Omega}(J*u_j-J*u)\,u\,dx
= \limsup_{j\to+\infty} 
\int_{\Omega}\big( J*(u_j-u)\big)\,u\,dx \\&\qquad\qquad\le 
\limsup_{j\to+\infty} \|u_j-u\|_{L^2(\Omega)}\|u\|_{L^2(\Omega)}= 0.\end{split}\end{equation}
So, from~\eqref{asrewrqrwqerq}, \eqref{8bis} and~\eqref{8ter}, we conclude that 
\begin{eqnarray*}&&\limsup_{j\to+\infty}\int_{\Omega}\big( u_j\,(J*u_j)-u\,(J*u) \big)\,dx 
\\&&\qquad\le \limsup_{j\to+\infty}\int_{\Omega} (u_j-u)\,(J*u_j)\,dx + 
 \limsup_{j\to+\infty}\int_{\Omega}(J*u_j-J*u)\,u\,dx =0.\end{eqnarray*} 
Also,
\begin{eqnarray*}&&\liminf_{j\to+\infty}
\iint_{Q_\Omega}\frac{|u_j(x)-u_j(y)|^2}{
|x-y|^{n+2s}}\,dx\,dy \ge
\iint_{Q_\Omega}\frac{|u(x)-u(y)|^2}{
|x-y|^{n+2s}}\,dx\,dy,\\ {\mbox{and }}&&
\liminf_{j\to+\infty}
\int_\Omega
\frac{\mu\,|u_j|^3}{3}\,dx\ge\int_\Omega
\frac{\mu\,|u|^3}{3}\,dx,\end{eqnarray*}
thanks to the Fatou Lemma. These inequalities imply that
$$ \liminf_{j\to+\infty} 
{\mathcal{E}}(u_j)\ge{\mathcal{E}}(u),$$
hence $u$ is the desired minimum.

Also, ${\mathcal{E}}(|u|)\le{\mathcal{E}}(u)$, so we can suppose in addition
that~$u$ is nonnegative.
Furthermore, $u$ is a solution of~\eqref{EQ} thanks to Lemma~\ref{90}.
\end{proof}

The claim in Theorem~\ref{EX56AJJ} now follows directly from the one
in Proposition~\ref{existence}.

\section{Qualitative properties and proof of Theorem~\ref{COP}}\label{SA:3}

The proof of Theorem~\ref{COP} is based on energy arguments,
by using the functional introduced in Section~\ref{SE:2}.
The details are the following:

\begin{proof}[Proof of Theorem~\ref{COP}] Assume that~$\sup_{\Omega}\sigma+\tau\le \lambda_s(\Omega)$.
Suppose, by contradiction, that there exists a nontrivial solution to~\eqref{EQ}.
Then, by Lemma~\ref{positive}, we have that~$u>0$ in~$\Omega$.

We observe that
\begin{equation}\label{0x9cv}
{\mbox{$\mu$ cannot vanish identically:}} \end{equation}
otherwise,
since~$(\sigma+\tau)^3\mu^{-2}\in L^1(\Omega)$, we would have that both~$\sigma$ and $\tau$
vanish identically as well, thus~$(-\Delta)^s u$ would
vanish identically in~$\Omega$, which would imply that~$u$
vanishes identically.

Therefore, using Lemma~\ref{young} (with~$v:=u$ and~$w:=u$)
and recalling~\eqref{0x9cv}, we see that
\begin{equation}\begin{split}\label{min}
&\int_\Omega (\sigma-\mu u)\,u^2\,dx +\int_{\Omega}\tau (J*u)u\,dx 
\le \int_\Omega (\sigma+\tau-\mu u)\,u^2\,dx
\\&\qquad \le\lambda_s(\Omega)\int_\Omega u^2\,dx
-\int_\Omega \mu u^3\,dx
< \lambda_s(\Omega)\int_\Omega u^2\,dx.\end{split}\end{equation}
Now, we test~\eqref{EQ} against~$u$ itself and we use~\eqref{min} to
see that
\begin{eqnarray*}
&& \lambda_s(\Omega)\|u\|_{L^2(\Omega)}^2\le
s\,(1-s)
\iint_{\R^{2n}}\frac{|u(x)-u(y)|^2}{
|x-y|^{n+2s}}\,dx\,dy \\&&\qquad =
\int_\Omega (\sigma-\mu u)\,u^2\,dx 
+\int_\Omega \tau (J*u)u\,dx
< \lambda_s(\Omega)\|u\|_{L^2(\Omega)}^2. \end{eqnarray*}
This is a contradiction and it establishes the first claim in
Theorem~\ref{COP}.

Now we show the second claim. For this,
we suppose~$\inf_{\Omega}\sigma\ge\lambda_s(\Omega)$ with strict inequality on a set of positive
measure and we remark that
it is enough to show that~$0$
is not a minimizer. To this goal, we take~$e$ to be the first
eigenfunction of~$(-\Delta)^s$ with Dirichlet datum
and~$\epsilon>0$. We recall that~$e>0$ in~$\Omega$ and it is bounded.
Then
\begin{eqnarray*}&&
{\mathcal{E}} (\epsilon e)\\&=&
\frac{\epsilon^2}{2} \left[s\,(1-s)
\iint_{\R^{2n}}\frac{|e(x)-e(y)|^2}{|x-y|^{n+2s}}\,dx\,dy
-\int_\Omega\sigma e^2\,dx-\int_\Omega \tau(J*e)e\,dx\right]
+\frac{\epsilon^3
}{3}\int_\Omega \mu\,|e|^3\,dx
\\&\le & \frac{\epsilon^2}{2} 
\int_\Omega (\lambda_s(\Omega)-\sigma) e^2\,dx 
+\frac{\epsilon^3}{3}\int_\Omega \mu\,|e|^3\,dx\\
&\le& -c_1\epsilon^2+c_2\epsilon^3,
\end{eqnarray*}
where
$$ c_1:=\frac{1}{2}
\int_\Omega (\sigma -\lambda_s(\Omega)) e^2\,dx
\ {\mbox{ and }} \
c_2:=\frac13\,\|\mu\|_{L^1(\Omega)}\|e_o\|_{L^\infty(\Omega)}^3.$$
Notice that~$c_1\in(0,+\infty)$.
So, if~$\epsilon$ is small, ${\mathcal{E}}(\epsilon e)<0={\mathcal{E}}(0)$,
showing that~$0$ is not a minimizer, hence the minimizer of
Proposition~\ref{existence}
is positive in~$\Omega$ and it provides a positive solution.
\end{proof}

\section{Adaptation to sparse resources and proof
of Theorem~\ref{CONG}}\label{SA:4}

The proof of Theorem~\ref{CONG} is based on a spectral analysis
and on the use of Theorem~\ref{COP}. The details are the following.

\begin{proof}[Proof of Theorem~\ref{CONG}] Since the domains are congruent,
we have that~$\lambda_s(\Omega_1)=\lambda_s(\Omega_2)$.
We claim that
\begin{equation}\label{RO}
\lambda_s (\Omega_1\cup\Omega_2)<\lambda_s(\Omega_1)=\lambda_s(\Omega_2)
.\end{equation}
To prove this, we take~$e_i$ to be the first eigenfunction
of~$\Omega_i$, for $i\in\{1,2\}$, normalized in such a way that~$\|e_i\|_{L^2(\R^n)}=\|e_i\|_{L^2(\Omega_i)}=1$.
Let~$e:=e_1+e_2$. Then
\begin{equation}\label{N-e}
\|e\|_{L^2(\Omega_1\cup\Omega_2)}^2 =
\|e\|_{L^2(\R^n)}^2=
\|e_1\|_{L^2(\R^n)}^2+\|e_2\|_{L^2(\R^n)}^2+
2\int_{\R^n} e_1(x)\,e_2(x)\,dx=2,\end{equation}
since the supports of~$e_1$ and~$e_2$ are disjoint. 
On the other hand, we know that~$e_i>0$ in~$\Omega_i$ (see e.g. Corollary~8
in~\cite{weak}), therefore
$$ \int_{\Omega_1}\left(\int_{\Omega_2} \frac{e_1(x)\,e_2(y)}{|x-y|^{n+2s}}
\,dy\right)\,dx >0.$$
Also, since~$e$ vanishes outside~$\Omega_1\cup\Omega_2$, we have that
\begin{eqnarray*}
&& \iint_{Q_{\Omega_1\cup\Omega_2}} \frac{|e(x)-e(y)|^2}{|x-y|^{n+2s}}\,dx\,dy=
\iint_{\R^{2n}} \frac{|e(x)-e(y)|^2}{|x-y|^{n+2s}}\,dx\,dy\\
&&\qquad= 
\iint_{\R^{2n}} \frac{|e_1(x)-e_1(y)|^2 + |e_2(x)-e_2(y)|^2 + 
2(e_1(x)-e_1(y))(e_2(x)-e_2(y))}{|x-y|^{n+2s}}\,dx\,dy
\\ &&\qquad= 
\iint_{Q_{\Omega_1}} \frac{|e_1(x)-e_1(y)|^2}{|x-y|^{n+2s}}\,dx\,dy
+ \iint_{Q_{\Omega_2}} \frac{|e_2(x)-e_2(y)|^2}{|x-y|^{n+2s}}\,dx\,dy
\\ &&\qquad\qquad+2\iint_{\R^{2n}} \frac{(e_1(x)-e_1(y))(e_2(x)-e_2(y))}{|x-y|^{n+2s}}\,dx\,dy
\\ &&\qquad= \frac{\lambda_s(\Omega_1)+\lambda_s(\Omega_2)}{s\,(1-s)}
+2\iint_{(\Omega_1\times\Omega_2)\cup (\Omega_2\times\Omega_1)
} \frac{(e_1(x)-e_1(y))(e_2(x)-e_2(y))}{|x-y|^{n+2s}}\,dx\,dy.
\end{eqnarray*}
Now we observe that
$$
\iint_{\Omega_1\times\Omega_2} 
\frac{(e_1(x)-e_1(y))(e_2(x)-e_2(y))}{|x-y|^{n+2s}}\,dx\,dy=
-\iint_{\Omega_1\times\Omega_2}
\frac{ e_1(x)e_2(y)}{|x-y|^{n+2s}}\,dx\,dy<0.$$
Similarly,
$$\iint_{\Omega_2\times\Omega_1}
\frac{(e_1(x)-e_1(y))(e_2(x)-e_2(y))}{|x-y|^{n+2s}}\,dx\,dy=
-\iint_{\Omega_2\times\Omega_1}
\frac{ e_1(y)e_2(x)}{|x-y|^{n+2s}}\,dx\,dy<0.$$
So we obtain that
$$ \iint_{Q_{\Omega_1\cup\Omega_2}} \frac{|e(x)-e(y)|^2}{|x-y|^{n+2s}}\,dx\,dy
<\frac{\lambda_s(\Omega_1)+\lambda_s(\Omega_2)}{s\,(1-s)}=
\frac{2\lambda_s(\Omega_1)}{s\,(1-s)}.$$
This and~\eqref{N-e} imply~\eqref{RO}, as desired.

{F}rom~\eqref{RO}, we can take
\begin{equation*}
\sigma \in \big( \lambda_s (\Omega_1\cup\Omega_2),\;\lambda_s(\Omega_1)\big)
=\big( \lambda_s (\Omega_1\cup\Omega_2),\;\lambda_s(\Omega_2)\big).
\end{equation*}
Then the claim in Theorem~\ref{CONG}
follows from
Theorem~\ref{COP}.
\end{proof}

It is worth to notice that Theorem~\ref{CONG} relies on
a purely nonlocal feature: indeed~\eqref{RO} fails
in the local case,
since
\begin{equation}\label{FAU}
\lambda_1(\Omega_1\cup\Omega_2)=\lambda_1(\Omega_1)=\lambda_1(\Omega_2).\end{equation}
Indeed, to prove~\eqref{FAU}, 
one may notice that~$e_1$ is an admissible competitor
for~$\lambda_1(\Omega_1\cup\Omega_2)$, hence~$\lambda_1(\Omega_1\cup\Omega_2)
\le\lambda_1(\Omega_1)$. On the other hand if~$\phi\in H^1_0(\Omega_1\cup\Omega_2)$,
then~$\phi_i:=\phi\chi_{\Omega_i}\in H^1_0(\Omega_i)$ for any~$i\in\{1,2\}$ and thus
$$ \frac{\int_{\Omega_1\cup\Omega_2}|\nabla \phi(x)|^2\,dx}{
\int_{\Omega_1\cup\Omega_2} \phi^2(x)\,dx} 
= \frac{\int_{\Omega_1}|\nabla \phi_1(x)|^2\,dx +
\int_{\Omega_2} |\nabla\phi_2(x)|^2\,dx}{
\int_{\Omega_1} \phi_1^2(x)\,dx
+\int_{\Omega_2} \phi_2^2(x)\,dx}.$$
Now we observe that if~$a_1$, $a_2$, $b_1$ and~$b_2$
are positive and such that~$\frac{a_1}{b_1}\le \frac{a_2}{b_2}$,
then
$$ \frac{a_1+a_2}{b_1+b_2} =
\frac{b_1(a_1+a_2)}{b_1(b_1+b_2)}\ge
\frac{a_1 b_1+ a_1 b_2}{b_1(b_1+b_2)} = \frac{a_1}{b_1}=
\min\left\{\frac{a_1}{b_1},\,\frac{a_2}{b_2} \right\} .$$
As a consequence
$$ \frac{\int_{\Omega_1\cup\Omega_2}|\nabla \phi(x)|^2\,dx}{
\int_{\Omega_1\cup\Omega_2} \phi^2(x)\,dx} \ge
\min \left\{
\frac{\int_{\Omega_1}|\nabla \phi_1(x)|^2\,dx}{
\int_{\Omega_1}\phi_1^2(x)\,dx},\,
\frac{\int_{\Omega_2}|\nabla \phi_2(x)|^2\,dx}{
\int_{\Omega_2}\phi_2^2(x)\,dx} \right\}\ge \lambda_1(\Omega),$$
which shows that~$\lambda_1(\Omega_1\cup\Omega_2)\ge\lambda_1(\Omega)$
and completes the proof of~\eqref{FAU}.

\section{Scaling arguments and proof of Proposition~\ref{NONAM}
and Theorem~\ref{EXT}}\label{SA:5}

The proof of Proposition~\ref{NONAM}
follows by a simple scaling argument, which we present
here for the sake of completeness:

\begin{proof}[Proof of Proposition~\ref{NONAM}] By scaling, we have that
\begin{equation}\label{SC}
\lambda_s(\Omega_r)= r^{-2s}
\lambda_s(\Omega).\end{equation} 
Also, by Theorem~\ref{COP}, a nontrivial
solution exists if and only if~$1>\lambda_s(\Omega_r)$. These
considerations imply the desired claim.
\end{proof}

The proof of Theorem~\ref{EXT} combines scaling arguments and spectral analysis
and it is presented here below.

\begin{proof}[Proof of Theorem~\ref{EXT}] Up to a translation, we may suppose that~$0\in\Omega$.
More precisely, we suppose that~$B_{a_1}\subset \Omega\subset B_{a_2}$,
for some~$a_2>a_1>0$. Then~$\lambda_s(B_{a_2})\le
\lambda_s(\Omega)\le \lambda_s(B_{a_1})$, that is, 
$$ c_1 \lambda_s(B_1)\le \lambda_s(\Omega) \le c_2\lambda_s(B_1),$$
for some~$c_2>c_1>0$.
Furthermore,
$$ \inf_{s\in (0,1]}\lambda_s(B_1) \ge c_3$$
for some~$c_3>0$. This follows, for instance, from\footnote{Regarding formula~(9) of~\cite{dyda} we remark that the map~$(0,1)\ni s\mapsto\gamma(s):=
(12 n+2s(16-2n))$ is monotone, therefore
$$\gamma(s)\ge \min\{\gamma(0),\gamma(1)\}=\min\{ 12n, 8n+32\}>0.$$
This and the continuity of the $\Gamma$-function in~$(0,+\infty)$
imply that the quantity in~(9) of~\cite{dyda} is bounded
from below uniformly in~$s\in(0,1]$.}
formulas~(9) and~(10)
in~\cite{dyda}. Furthermore
$$ \sup_{s\in (0,1]}\lambda_s(B_1) \le c_4.$$
This may be checked by 
fixing $g\in C^\infty_0(B_{a_1})$ with~$\|g\|_{L^2(\R^n)}=1$,
and using that
$$ \lambda_s(\Omega)\le 
s\,(1-s)\, \iint_{Q_\Omega}\frac{|g(x)-g(y)|^2}{
|x-y|^{n+2s}}\,dx\,dy\le c_5 \|g\|_{C^2(\R^n)},$$
for some~$c_5>0$.

The above consideration and the scaling property~\eqref{SC}
give that
$$ c_1\,c_3 \,r^{-2s}\le
\lambda_s(\Omega_r)\le c_2\,c_4 \,r^{-2s},$$
for any~$s\in(0,1]$.

Now we fix~$s<S\in(0,1]$ and we set
$$ \underline r:= \left( \frac{c_1\,c_3}{c_2\,c_4}\right)^{\frac{1}{2(S-s)}}
\;{\mbox{ and }}\;\overline r:= \left( \frac{c_2\,c_4}{c_1\,c_3}\right)^{\frac{1}{2(S-s)}}.$$
Then, if~$r\in(0,\underline r)$ we have that
$$ \lambda_S(\Omega_r)-\lambda_s(\Omega_r)\ge c_1\,c_3\,r^{-2S}-
c_2\,c_4\,r^{-2s}>0,$$
thus we can find~$\sigma_r$ in the interval~$\big( \lambda_s(\Omega_r),\,
\lambda_S(\Omega_r)\big)$. Moreover, we can also find $\tau_r$ such that 
$$ \lambda_s(\Omega_r) <\sigma_r \le \sigma_r +\tau_r 
<\lambda_S(\Omega_r). $$
{F}rom
Theorem~\ref{COP}, we have that in this case equation~\eqref{ER2}
has a nontrivial solution, while~\eqref{ER1}
only has the trivial solution.

Viceversa, if~$r\in(\overline r,+\infty)$ then~$\lambda_S(\Omega_r)-\lambda_s(\Omega_r)<0$,
hence we can find~$\sigma_r$ in the interval~$\big( \lambda_S(\Omega_r),\,
\lambda_s(\Omega_r)\big)$ and $\tau_r$ such that 
$$ \lambda_S(\Omega_r) <\sigma_r \le \sigma_r +\tau_r 
<\lambda_s(\Omega_r). $$
In this case,
Theorem~\ref{COP} gives that~\eqref{ER1}
has a nontrivial solution, while~\eqref{ER2}
has only the trivial solution.
\end{proof}

\section{Fitting the resources and proof of Lemma~\ref{easy}
and Theorem~\ref{abundance}}\label{SA:6}

The proof of Lemma~\ref{easy} is a simple maximum principle,
whose details are presented here below for completeness:

\begin{proof}[Proof of Lemma~\ref{easy}]
Suppose by contradiction that
there exists~$x_o\in\Omega$ such that~$0<
\max_{\R^n} u-\|\sigma\|_{L^\infty(\Omega)}-\tau =u(x_o)-\|\sigma\|_{L^\infty(\Omega)}-\tau$. 
Notice that, using~\eqref{unoint}, 
$$ (J*u)(x_o)=\int_{\R^n} J(x_o-y)u(y)\,dy \le u(x_o)\int_{\R^n} J(z)\,dz = u(x_o).$$ 
Then
$$ 0\le(-\Delta)^{s} u(x_o) =
(\sigma(x_o)- u(x_o))\,u(x_o)+ \tau\,(J*u)(x_o) \le 
(\sigma(x_o)+\tau- u(x_o))\,u(x_o) <0,$$
which is a contradiction.
\end{proof}

Now we show that~$u$ always fits the ``abundant'' 
resources (up to a multiplicative
constant):

\begin{proposition}\label{abundance:2}
Let~$R>r>0$.
Let~$\Omega$ be a bounded Lipschitz domain, with~$\overline{B_{R}}\subset\Omega$.
Let~$u$ be the minimal
solution of
\begin{equation*}
\left\{ \begin{matrix}(-\Delta)^{s} u =
(\sigma- u)\,u +\tau\,(J*u) &{\mbox{ in }}\Omega,
\\u=0 & {\mbox{ outside }}\Omega,\\
u\ge0&{\mbox{ in }}\R^n,\end{matrix}\right.
\end{equation*}
according to Proposition~\ref{existence}.

Then, there exist~$c_o\in(0,1)$ only depending on~$n$, $s$,
$R$ and~$r$, and~$M_o>0$ only 
depending on~$n$, $s$ and~$R$,
such that if~$M\ge M_o$ and $\sigma\ge M$ 
in~$B_R$, then~${\mathcal{E}}(u)<0$ and~$u\ge c_oM$ in~$B_r$.
\end{proposition}

\begin{proof} 
We take~$e_o$ to be the first Dirichlet eigenfunction of~$B_{R}$.
Then we have
\begin{eqnarray*}
{\mathcal{E}}(e_o)&=&\frac{s\,(1-s)}{2} \iint_{\R^{2n}}
\frac{|e_o(x)-e_o(y)|^2}{
|x-y|^{n+2s}}\,dx\,dy
+\int_\Omega \frac{|e_o|^3}{3}
-\frac{\sigma\,e_o^2}{2}-\frac{\tau\,e_o\,(J*e_o)}{2}\,dx \\
&=& \frac{\lambda_s(B_{R})}{2} \int_{B_{R}} e_o^2\,dx
+\int_{B_R} \frac{|e_o|^3}{3}
-\frac{\sigma\,e_o^2}{2}-\frac{\tau\,e_o\,(J*e_o)}{2}\,dx\\
&\le&
\frac{\lambda_s(B_{R})-M}{2} \int_{B_{R}} e_o^2\,dx
+\frac{|B_R|\,\|e_o\|_{L^\infty(\R^n)}^3}{3}-\int_{B_R}\frac{\tau\,e_o\,(J*e_o)}{2}\,dx 
\end{eqnarray*}
The latter quantity is negative if~$M\ge M_o$,
for large values of~$M_o$,
therefore the energy of the minimizer~$u$ is negative
and~$u$ is not the trivial function.

Consequently, from Proposition~\ref{existence} and Lemma~\ref{positive}, we can define
$$ \iota := \inf_{B_R} u>0.$$
In particular, if~$\eta \in (0, \,\iota\,\|e_o\|^{-1}_{L^\infty(\R^n)})$
we have that~$\eta e_o\le u$. So we take the first~$\eta$ for
which a contact point in $\Omega$ occurs (of course, if~$\eta e_o\le u$
for all~$\eta>0$,
we obtain the desired result by taking~$\eta$ as large as we wish,
hence we can assume that such contact point exists).
That is, we have that~$\eta e_o\le u$ and there exists~$\bar x\in\Omega$
such that~$\eta e_o(\bar x)= u(\bar x)$. Since~$e_o$ vanishes outside~$B_R$,
we have that~$\bar x\in B_R$. Therefore
\begin{eqnarray*}
&&0\ge (-\Delta)^s ( u-\eta e_o)(\bar x) =(\sigma(\bar x)-u(\bar x))u(\bar x) 
+\tau\,(J*u)(\bar x) -
\eta \lambda_s(B_R)\,e_o(\bar x) \\&&\qquad\qquad\ge 
(\sigma(\bar x)-u(\bar x))u(\bar x) -\lambda_s(B_R)\,u(\bar x) .\end{eqnarray*}
Accordingly,
$$ 0\ge M u(\bar x) -u^2(\bar x) -\lambda_s(B_R)\,u(\bar x)\ge
\frac{M}{2} u(\bar x) -u^2(\bar x),$$
as long as~$M\ge M_o$ and~$M_o$ is large enough.
This says that
$$ \frac{M}{2} \le u(\bar x)=\eta e_o(\bar x)\le \eta \,\|e_o\|_{L^\infty(\R^n)}.$$
In particular~$\eta\ge M/(2 \,\|e_o\|_{L^\infty(\R^n)})$ and therefore, for any~$x\in B_r$,
\begin{equation*}
u(x)\ge \eta e_o(x) \ge
\frac{\inf_{B_r} e_o}{2\,\|e_o\|_{L^\infty(\R^n)}}\,M.\qedhere\end{equation*}
\end{proof}

Now, Theorem~\ref{abundance} follows plainly from
Proposition~\ref{abundance:2}.

\section{Fitting the resources in a nonlocal setting
and proof of Theorem~\ref{BOVI}}\label{SA:6:BIS}

Now we prove Theorem~\ref{BOVI}, by exploiting
a result in~\cite{Di_Sa_Val}, joined to a minimization argument.

More precisely, we make use of
Theorem~1.1 in~\cite{Di_Sa_Val}, which we state here for the
convenience of the reader:

\begin{theorem}\label{AHJ7uHAa13HA}
Fix $k \in \N$. Then, given any function~$ f \in C^k(B_2)$ and any~$\varepsilon>0$,
we can find~$R_\varepsilon>2$ and a function~$u_\varepsilon\in C^s_0(
B_{R_\varepsilon})$ such that
\begin{eqnarray*}
&& (-\Delta)^s u_\varepsilon =0 {\mbox{ in }} B_2\\
{\mbox{and }}&& \| u_\varepsilon -f\|_{C^k(\overline{B_2})}\le
\varepsilon.\end{eqnarray*}
\end{theorem}

The details of the proof of Theorem~\ref{BOVI} now go as follows:

\begin{proof}[Proof of Theorem~\ref{BOVI}] First of all, we use
Theorem~\ref{AHJ7uHAa13HA}
to find a function~$w_\varepsilon$ and a radius~$R_\varepsilon>2$ such that
\begin{equation}
\label{KKJASFOF}\begin{split}
& (-\Delta)^s w_\varepsilon = 0 \qquad{\mbox{ in $B_2$,}}\\
& w_\varepsilon =0 \qquad{\mbox{ in $\R^n\setminus B_{R_\varepsilon}$}},\\
{\mbox{and }}\;& \|w_\varepsilon-\mu^{-1}\sigma\|_{C^k(\overline{B_2})}
\le\varepsilon.\end{split}
\end{equation}
Let
\begin{equation}\label{7uhs1a2}
W_\varepsilon:=|w_\varepsilon|\;{\mbox{ and }}\;\sigma_\varepsilon:=
\mu w_\varepsilon.\end{equation} Notice that
$$  \|\sigma_\varepsilon-\sigma\|_{C^k(\overline{B_1})} =
 \|\mu (w_\varepsilon-\mu^{-1}\sigma)\|_{C^k(\overline{B_1})}
\le C_k \,\|w_\varepsilon-\mu^{-1}\sigma\|_{C^k(\overline{B_1})}
\le C_k \varepsilon,$$
for some~$C_k>0$, possibly depending on~$\|\mu\|_{C^k(\overline{B_1})}$,
and this proves~\eqref{LAJ6:OV:CC} (up to renaming~$\varepsilon$).

Moreover, if~$x\in \overline{B_2}$,
$$ w_\varepsilon\ge \mu^{-1}\sigma-
\|w_\varepsilon-\mu^{-1}\sigma\|_{L^\infty(\overline{B_2})}\ge
\inf_{\overline{B_2}} \mu^{-1}\sigma-\varepsilon\ge0,$$
if we take~$\varepsilon>0$ small enough,
therefore
\begin{equation}\label{YUJA:A}
{\mbox{$W_\varepsilon=w_\varepsilon$ in~$\overline{B_2}$.}}\end{equation}
Accordingly, for any~$x\in B_1$,
\begin{eqnarray*} \int_{\R^n} \frac{W_\varepsilon(x+y)+
W_\varepsilon(x-y)-2W_\varepsilon(x)}{|y|^{n+2s}}\,dy
&\ge& \int_{\R^n} \frac{w_\varepsilon(x+y)+
w_\varepsilon(x-y)-2W_\varepsilon(x)}{|y|^{n+2s}}\,dy\\&=&
\int_{\R^n} \frac{w_\varepsilon(x+y)+
w_\varepsilon(x-y)-2w_\varepsilon(x)}{|y|^{n+2s}}\,dy\end{eqnarray*}
and thus
$$ -(-\Delta)^s W_\varepsilon(x)\ge0$$
for any~$x\in B_1$.
As a consequence,
\begin{equation}\label{6fgaq1a}
f_\varepsilon(x):= \tau (J*W_\varepsilon)(x) -(-\Delta)^s W_\varepsilon(x)\ge0
\end{equation}
for any~$x\in B_1$.

By~\eqref{YUJA:A}, we get that~$
(-\Delta)^s W_\varepsilon\in L^\infty(B_1)$, and consequently
\begin{equation}\label{6fgaq1a:BIS}
f_\varepsilon\in L^\infty(B_1).
\end{equation}
Now we introduce the energy functional
$$ {\mathcal{G}}(v):=
\frac{s\,(1-s)}{2}\iint_{\R^{2n}} \frac{|v(x)-v(y)|^2}{|x-y|^{n+2s}}\,dx\,dy
+\int_{B_1} \frac{\mu |v|^3}{3}+\frac{\sigma_\varepsilon v^2}{2}
-f_\varepsilon \,v -\frac{\tau v\,(J*v)}{2}\,dx$$
and we aim to minimize~${\mathcal{G}}$ among all the functions
that vanish outside~$B_1$. For this, we observe that~$ {\mathcal{G}}(0)=0$
and we take a minimizing sequence~$v_j$, namely
\begin{equation}\label{9iaHHJJA:0}
\lim_{j\to+\infty} {\mathcal{G}}(v_j)=\inf
{\mathcal{G}}(v),\end{equation}
where the infimum is taken among the functions~$v$
such that~$v=0$ in~$\R^n\setminus B_1$.
We observe that, by~\eqref{LAJ6:OV:CC},
we know that
$$ \inf_{\overline{B_1}}\sigma_\varepsilon>0.$$
Also, by Lemma~\ref{young},
$$ \int_{B_1}v_j(x)(J*v_j)(x)\,dx \le \|v_j\|_{L^2(B_1)}^2$$
and, by~\eqref{7bis} (used here with~$\sigma=1$),
$$ \frac{(1+\tau)\,v_j^2}{2}
\le \frac{\mu\,|v_j|^3}{3}+\frac{(1+\tau)^3}{6 \mu^2}.$$
Using these considerations,
we find that
\begin{eqnarray*}
&& \int_{B_1} 
f_\varepsilon \,v_j +\frac{\tau v_j\,(J*v_j)}{2}
-\frac{\mu |v_j|^3}{3}
\,dx 
\le
\int_{B_1}
\frac{ f_\varepsilon^2+v_j^2}{2} +\frac{\tau v_j^2}{2}
-\frac{\mu |v_j|^3}{3}
\,dx \\ &&\qquad\qquad\le 
\int_{B_1}
\frac{ f_\varepsilon^2}{2}+  
\frac{(1+\tau)^3}{6 \mu^2} 
\,dx\le C_\varepsilon,
\end{eqnarray*}
for some~$C_\varepsilon >0$ that does not depend on~$j$.
As a consequence,
$$ {\mathcal{G}}(v_j)\ge
\frac{s\,(1-s)}{2}\iint_{\R^{2n}} \frac{|v_j(x)-v_j(y)|^2}{|x-y|^{n+2s}}\,dx\,dy
+\int_{B_1} \frac{\sigma_\varepsilon v_j^2}{2}\,dx -C_\varepsilon.$$
This gives that~$v_j$ is precompact in~$L^2(B_1)$
(see e.g. Theorem~7.1 in~\cite{guida}) and so we may suppose,
up to a subsequence, that it converges to some~$v_\star$
in~$L^2(B_1)$ and a.e. in~$\R^n$, with~$v_\star=0$ outside~$B_1$.

Therefore, by Fatou Lemma,
\begin{equation}\label{9iaHHJJA:1}
\begin{split}
&\liminf_{j\to+\infty}
\frac{s\,(1-s)}{2}\iint_{\R^{2n}} \frac{|v_j(x)-v_j(y)|^2}{|x-y|^{n+2s}}\,dx\,dy
+\int_{B_1} \frac{\mu |v_j|^3}{3}+\frac{\sigma_\varepsilon v_j^2}{2}
\\ &\qquad\ge
\frac{s\,(1-s)}{2}\iint_{\R^{2n}} \frac{|v_\star(x)-v_\star(y)|^2}{|x-y|^{n+2s}}\,dx\,dy
+\int_{B_1} \frac{\mu |v_\star|^3}{3}+\frac{\sigma_\varepsilon v_\star^2}{2}
.\end{split}
\end{equation}
Also, by weak convergence in~$L^2(B_1)$,
\begin{equation}\label{9iaHHJJA:2}
\lim_{j\to+\infty}\int_{B_1} f_\varepsilon \,v_j\,dx=
\int_{B_1} f_\varepsilon \,v_\star\,dx.
\end{equation}
In addition, by Lemma~\ref{young},
\begin{eqnarray*}
&& \left| \int_{B_1} v_\star (J*v_\star) -v_j(J*v_j)\,dx\right|\\
&\le&
\int_{B_1}| v_\star-v_j| \,|J*v_\star|\,dx+
\int_{B_1} |v_j|\, |J*(v_\star-v_j)|\,dx
\\ &\le& \|v_\star-v_j\|_{L^2(B_1)}\, \|v_\star\|_{L^2(B_1)}+
\|v_\star-v_j\|_{L^2(B_1)}\,\|v_j\|_{L^2(B_1)}
\end{eqnarray*}
that are infinitesimal as~$j\to+\infty$. Using this,
\eqref{9iaHHJJA:0}, \eqref{9iaHHJJA:1}
and~\eqref{9iaHHJJA:2}, we obtain that~$v_\star$ is a minimizer
for~${\mathcal{G}}$.

Since~${\mathcal{G}}(|v|)\le {\mathcal{G}}(v)$, due to~\eqref{6fgaq1a},
we may also suppose that
\begin{equation}\label{8ujx52ysuwdgh}
{\mbox{$v_\star$ is nonnegative.}}
\end{equation}
The minimization property of~$v_\star$ gives that
$$ (-\Delta)^s v_\star+ \mu v_\star^2 +
\sigma_\varepsilon v_\star
-f_\varepsilon -\tau (J*v_\star)=0$$
in~$B_1$. Hence, we define
$$ u_\varepsilon := W_\varepsilon + v_\star$$
and, recalling~\eqref{6fgaq1a}, we find that
\begin{equation}\label{9wdq126qsyc}
\begin{split}
& -(-\Delta)^s u_\varepsilon +(\sigma_\varepsilon -\mu\, u_\varepsilon)\,
u_\varepsilon +\tau (J*u_\varepsilon)\\
=\;& -(-\Delta)^s W_\varepsilon
+\mu v_\star^2 +
\sigma_\varepsilon v_\star
-f_\varepsilon -\tau (J*v_\star)+\sigma_\varepsilon W_\varepsilon +
\sigma_\varepsilon v_\star\\
&\qquad-\mu W_\varepsilon^2 -\mu v_\star^2 -2\mu
W_\varepsilon v_\star +\tau(J*W_\varepsilon)+\tau(J*v_\star)\\
=\;& 2\sigma_\varepsilon v_\star
+\sigma_\varepsilon W_\varepsilon -\mu W_\varepsilon^2 -2\mu
W_\varepsilon v_\star .
\end{split}\end{equation}
in~$B_1$. Now we recall~\eqref{7uhs1a2}
and~\eqref{YUJA:A} and we find that, in~$B_1$,
$$ W_\varepsilon=w_\varepsilon=\mu^{-1}\sigma_\varepsilon.$$
Hence we insert this identity into~\eqref{9wdq126qsyc}
and we conclude that
$$ -(-\Delta)^s u_\varepsilon +(\sigma_\varepsilon -\mu\, u_\varepsilon)\,
u_\varepsilon +\tau (J*u_\varepsilon)=
2\sigma_\varepsilon v_\star
+\mu^{-1}\sigma_\varepsilon^2
-\mu \mu^{-2}\sigma_\varepsilon^2
-2\sigma_\varepsilon v_\star=0$$
in $B_1$,
which establishes~\eqref{LAJ6:OV:1}.

Also, by~\eqref{KKJASFOF}, we have that both~$W_\varepsilon$
and~$v_\star$ vanish outside~$B_{R_\varepsilon}$, and this establishes~\eqref{LAJ6:OV:2}.
Finally, by~\eqref{7uhs1a2}
and~\eqref{8ujx52ysuwdgh},
$$ u_\varepsilon\ge W_\varepsilon \ge w_\varepsilon
=\mu^{-1}\sigma_\varepsilon,$$
which proves~\eqref{LAJ6:OV:3}.
\end{proof}

\section{Beating the resources and proof of Theorem~\ref{BEAT0}}\label{SA:7}

The proof of Theorem~\ref{BEAT0} is based on a contradiction and limit argument.

\begin{proof}[Proof of Theorem~\ref{BEAT0}] Let~$u_m$ be the solution of~\eqref{df}
provided by Proposition~\ref{existence}. If the desired claim were false,
we would have that~$u_m\le\sigma_m$.
Then
$$ \big| (-\Delta)^{s} u_m\big| =
(\sigma_m- u_m)\,u_m  \le \sigma_m u_m
\le \|\sigma_m\|_{L^\infty(\Omega)} \| u_m\|_{L^\infty(\Omega)}.$$
Hence, using Lemma~\ref{easy} with $\tau=0$,
$$ \big| (-\Delta)^{s} u_m\big|\le \|\sigma_m\|_{L^\infty(\Omega)}^2
\le 
\big( \|\sigma_0\|_{L^\infty(\Omega)}+1\big)^2.$$
Notice that the latter quantity does not depend on~$m$.
Thus, by fractional elliptic regularity (see e.g. Proposition~1.1
in~\cite{ros-serra} and Lemma~4.3 in~\cite{silv-app})
we have that~$u_m$ converges uniformly in~$\Omega$ to some~$u_0$
as~$m\to0$, and~$u_0$ solves
$$ (-\Delta)^{s} u_0 =
(\sigma_0- u_0)\,u_0$$
in $\Omega$. By Theorem~\ref{abundance}, we know that~$u_0>0$
in~$B_r$. In particular~$u_0$ is not the trivial solution,
and so~$u_0>0$, thanks to Lemma~\ref{positive}. Then we have
$$ 0<u_0(x_0)=\lim_{m\to0} u_m(x_0)
\le \lim_{m\to0}\sigma_m(x_0)=\sigma_0(x_0)=0,$$
which is a contradiction.
\end{proof}

\section{Periodic solutions and proof of 
Theorems~\ref{0o56:TH} and~\ref{0o56:TH:2}}\label{SA:PP}

To prove Theorem~\ref{0o56:TH}, we consider an 
auxiliary minimization problem. The functional is tailored in order to
be compatible with integer translations and produce
solutions of~\eqref{EQ:PER} via an Euler-Lagrange equation,
tested against periodic test functions.

Here, we assume that~$J$ is supported in some ball~$B_\rho$ and
we let
\begin{equation}\label{defQ}
Q:= \left(-\frac{1}{2}, \frac{1}{2}\right]^n.\end{equation} 
We define the energy functional
\begin{eqnarray*}
{\mathcal{F}}(v)&:=&
\frac{s\,(1-s)}{2}\iint_{\R^n\times Q}
\frac{|v(x)-v(y)|^2}{|x-y|^{n+2s}}\,dx\,dy+
\int_{Q} \frac{\mu |v|^3}{3}-\frac{\sigma v^2}{2}\,dx\\
&& -\frac{\tau}{2} \iint_{\R^n\times Q} J(x-y)\,v(x)\,v(y)\,dx\,dy.
\end{eqnarray*}
Then we consider the space~$X$ of functions~$v\in L^2(Q)$,
with~$v(x+k)=v(x)$ for any~$k\in \Z^n$ and a.e.~$x\in\R^n$. We have that~${\mathcal{F}}$
attains a minimum in~$X$, according to the following result:

\begin{lemma}\label{N:MI}
There exists~$v_*\in X$ such that~${\mathcal{F}}(v_*)\le
{\mathcal{F}}(v)$ for every~$v\in X$.
\end{lemma}

\begin{proof} First of all, we notice that~${\mathcal{F}}(0)=0$,
so we take a minimizing sequence~$v_j\in X$ such that
\begin{equation}\label{TA:1}
\lim_{j\to+\infty} {\mathcal{F}}(v_j) =\inf_X {\mathcal{F}}\end{equation}
and we may suppose that
\begin{equation}\label{KA:1}
{\mathcal{F}}(v_j)\le0.\end{equation} 
Our goal is to obtain estimates that are uniform in~$j$.

Letting~$w_j:= |v_j|\chi_{B_{\rho+\sqrt{n}}}$ and recalling Lemma~\ref{young}, we see that
\begin{eqnarray*}
&& \left|\iint_{\R^n\times Q} J(x-y)\,v_j(x)\,v_j(y)\,dx\,dy\right|
\le \iint_{B_{\rho+\sqrt{n}}\times Q} 
J(x-y)\,|v_j(x)|\,|v_j(y)|\,dx\,dy
\\ &&\qquad
\le \iint_{\R^{2n}} J(x-y)\,w_j(x)\,w_j(y)\,dx\,dy\le
\| w_j\|_{L^2(B_{\rho+\sqrt{n}})}^2 \le C\, \| v_j\|_{L^2(Q)}^2,\end{eqnarray*}
for some~$C>0$, possibly depending on~$\rho$ and~$n$.
Hence, 
\begin{equation}\label{Q1:A}
\begin{split}
& \int_{Q} \frac{\mu |v_j|^3}{3}-\frac{\sigma v_j^2}{2}\,dx
-\frac{\tau}{2} \iint_{\R^n\times Q} J(x-y)\,v_j(x)\,v_j(y)\,dx\,dy
\\ &\qquad\ge \int_{Q} \frac{\mu |v_j|^3}{3}-\frac{\sigma v_j^2}{2}\,dx
-\frac{C\tau}{2} \int_{Q} v_j^2\,dx
.\end{split}\end{equation}
Using this and~\eqref{7bis} (with~$C\tau$ in the place of~$\tau$), we get
\begin{eqnarray*}
&& \int_{Q} \frac{\mu |v_j|^3}{3}-\frac{\sigma v_j^2}{2}\,dx
-\frac{\tau}{2} \iint_{\R^n\times Q} J(x-y)\,v_j(x)\,v_j(y)\,dx\,dy
\\&& \qquad\ge-\int_{Q} \frac{(\sigma+C\tau)^3}{6\mu^2}\,dx
=: -\kappa,\end{eqnarray*}
where~$\kappa>0$ depends on~$\sigma$, $\tau$, $\mu$, $\rho$ and~$n$. As a consequence
of this and~\eqref{KA:1}, we obtain
\begin{equation}\label{LA:1}
\frac{s\,(1-s)}{2}\iint_{\R^n\times Q}
\frac{|v_j(x)-v_j(y)|^2}{|x-y|^{n+2s}}\,dx\,dy\le \kappa.\end{equation}
In addition, utilizing~\eqref{KA:1} and~\eqref{Q1:A}, we have that
$$ \int_{Q} \frac{\mu |v_j|^3}{3}-\frac{\sigma v_j^2}{2}\,dx
-\frac{C\tau}{2} \int_{Q} v_j^2\,dx\le0,$$
and so, by H\"older Inequality,
$$ \int_{Q} \frac{\mu |v_j|^3}{3} \le 
\frac{\|\sigma\|_{L^\infty(Q)}+C\tau}{2} \left(\int_{Q} v_j^3\,dx\right)^{2/3}.$$
Accordingly, $\| v_j\|_{L^3(Q)}$ is bounded uniformly in~$j$
and therefore $\| v_j\|_{L^2(Q)}$ is also bounded uniformly in~$j$.

{F}rom this and~\eqref{LA:1},
it follows that~$v_j$ is precompact in~$L^2(Q)$
(see e.g. Theorem~7.1 in~\cite{guida}). Thus, up to a subsequence,
we may assume that~$v_j\to v_*$ in~$L^2(Q)$ and a.e. in~$Q$
(and thus, by periodicity, a.e. in~$\R^n$), as~$j\to+\infty$.
Notice also that~$v_*$ is periodic, since so is~$v_j$. 
This gives that~$v_*\in X$.
Furthermore, using the convergence of~$v_j$ and Fatou Lemma,
$$ \liminf_{j\to+\infty}
\frac{s\,(1-s)}{2}\iint_{\R^n\times Q}
\frac{|v_j(x)-v_j(y)|^2}{|x-y|^{n+2s}}\,dx\,dy\ge
\frac{s\,(1-s)}{2}\iint_{\R^n\times Q}
\frac{|v_*(x)-v_*(y)|^2}{|x-y|^{n+2s}}\,dx\,dy,$$
\begin{eqnarray*}
&& \liminf_{j\to+\infty}\int_{Q} \frac{\mu |v_j|^3}{3}\,dx\ge
\int_{Q} \frac{\mu |v_*|^3}{3}\,dx\\
{\mbox{and }}&& \lim_{j\to+\infty} \int_{Q}\frac{\sigma v_j^2}{2}\,dx
=\int_{Q}\frac{\sigma v_*^2}{2}\,dx .\end{eqnarray*}
Moreover,
\begin{eqnarray*}
&& \left|\iint_{\R^n\times Q} J(x-y)\,v_j(x)\,v_j(y)\,dx\,dy
- \iint_{\R^n\times Q} J(x-y)\,v_*(x)\,v_*(y)\,dx\,dy\right|\\
&\le&
\iint_{B_{\rho+\sqrt{n}}\times Q} J(x-y)\,\big|v_j(x)\big|\,\big|v_j(y)-v_*(y)\big|\,dx\,dy
\\&&\qquad+\iint_{B_{\rho+\sqrt{n}}\times Q} J(x-y)\,\big|v_j(x)-v_*(x)\big|\,\big|v_*(y)\big|\,dx\,dy\\
&\le& C\,\Big(\|v_j\|_{L^2(Q)}\, \|v_j-v_*\|_{L^2(Q)}+\|v_j-v_*\|_{L^2(Q)}\, \|v_*\|_{L^2(Q)}\Big),
\end{eqnarray*}
thanks to Lemma~\ref{young}, and the latter quantity is infinitesimal
as~$j\to+\infty$. These considerations and~\eqref{TA:1} give that
$$ {\mathcal{F}}(v_*) =\inf_X {\mathcal{F}},$$
so the desired result follows.
\end{proof}

Now we can complete the proof of Theorem~\ref{0o56:TH}
by considering the minimizer produced by
Lemma~\ref{N:MI} and by checking that periodic perturbations
indeed give~\eqref{EQ:PER} as Euler-Lagrange equation.

\begin{proof}[Proof of Theorem~\ref{0o56:TH}] Let~$v_*$ be as in Lemma~\ref{N:MI} and~$u:=|v_*|$.
Then
\begin{equation}\label{45:A} {\mathcal{F}}(u)\le {\mathcal{F}}(v_*)\le
{\mathcal{F}}(v)\qquad{\mbox{ for every~$v\in X$.}}\end{equation}
Now we take~$\psi\in C^\infty_0(Q)$ and we consider its periodic
extension in~$\R^n$, that is
$$ \phi(x):= \sum_{k\in\Z^n} \psi(x+k).$$
Using~$v:=u+\epsilon \phi$ as test function in~\eqref{45:A}, we obtain
that
\begin{equation}\label{AS:IO:1}
\begin{split}
& {s\,(1-s)}\iint_{\R^n\times Q}
\frac{\big(u(x)-u(y)\big)\big(\phi(x)-\phi(y)\big)}{|x-y|^{n+2s}}\,dx\,dy+
\int_{Q} \mu u^2\phi-\sigma u\phi\,dx\\
&\qquad -\frac{\tau}{2} \iint_{\R^n\times Q} J(x-y)\,u(x)\,\phi(y)\,dx\,dy
-\frac{\tau}{2} \iint_{\R^n\times Q} J(x-y)\,\phi(x)\,u(y)\,dx\,dy=0.
\end{split}\end{equation}
Now we write
$$ \R^n = \bigcup_{k\in\Z^n} (Q+k)$$
and thus, using the substitutions~$\tilde x:=x-k$ and~$\tilde y:=y-k$,
\begin{equation}\label{AS:IO:2}
\begin{split}
& \iint_{\R^{2n}}
\frac{\big(u(x)-u(y)\big)\big(\psi(x)-\psi(y)\big)}{|x-y|^{n+2s}}\,dx\,dy
\\ &\qquad=\sum_{k\in\Z^n} 
\iint_{\R^{n}\times (Q+k)}
\frac{\big(u(x)-u(y)\big)\big(\psi(x)-\psi(y)\big)}{|x-y|^{n+2s}}\,dx\,dy
\\ &\qquad =\sum_{k\in\Z^n} 
\iint_{\R^{n}\times Q}
\frac{\big(u(\tilde x+k)-u(\tilde y+k)\big)
\big(\psi(\tilde x+k)-\psi(\tilde y+k)\big)}{
|\tilde x-\tilde y|^{n+2s}}\,d\tilde x\,d\tilde y 
\\ &\qquad =
\sum_{k\in\Z^n}
\iint_{\R^{n}\times Q}
\frac{\big(u(\tilde x)-u(\tilde y)\big)
\big(\psi(\tilde x+k)-\psi(\tilde y+k)\big)}{
|\tilde x-\tilde y|^{n+2s}}\,d\tilde x\,d\tilde y
\\ &\qquad =
\iint_{\R^{n}\times Q}
\frac{\big(u(\tilde x)-u(\tilde y)\big)\,\sum_{k\in\Z^n}
\big(\psi(\tilde x+k)-\psi(\tilde y+k)\big)}{
|\tilde x-\tilde y|^{n+2s}}\,d\tilde x\,d\tilde y
\\ &\qquad =
\iint_{\R^n\times Q}
\frac{\big(u(\tilde x)-u(\tilde y)\big)\big(\phi(\tilde x)-\phi(\tilde y)\big)}{|\tilde x-\tilde y|^{n+2s}}\,d\tilde x\,d\tilde y.
\end{split}\end{equation}
Similarly,
\begin{equation}\label{AS:IO:3}
\begin{split}
& \iint_{\R^{2n}} J(x-y)\,\psi(x)\,u(y)\,dx\,dy
=\sum_{k\in\Z^n} \iint_{\R^{n}\times (Q+k)} J(x-y)\,\psi(x)\,u(y)\,dx\,dy\\
&\qquad= \sum_{k\in\Z^n} \iint_{\R^{n}\times Q} J(\tilde x-\tilde y)\,
\psi(\tilde x+k)\,u(\tilde y)\,dx\,dy
=
\iint_{\R^n\times Q} J(\tilde x-\tilde y)\,\phi(\tilde x)\,u(\tilde y)\,d\tilde x\,d\tilde y.
\end{split}\end{equation}
So, we insert \eqref{AS:IO:2} and~\eqref{AS:IO:3}
into~\eqref{AS:IO:1} and we obtain that
\begin{eqnarray*}
&& {s\,(1-s)}\iint_{\R^{2n}}
\frac{\big(u(x)-u(y)\big)\big(\psi(x)-\psi(y)\big)}{|x-y|^{n+2s}}\,dx\,dy+
\int_{\R^n} \mu u^2\psi-\sigma u\psi\,dx\\
&&\qquad -{\tau} \iint_{\R^{2n}} J(x-y)\,u(x)\,\psi(y)\,dx\,dy=0.
\end{eqnarray*}
This gives that~$u$ is a solution of the desired equation in~$Q$
(and thus in the whole of~$\R^n$, by periodicity).

We also claim that
\begin{equation}\label{0o56:BIS}
{\mbox{$u>0$ in $\R^n$.}}
\end{equation}
The proof is by contradiction: if there exists~$x_o$ for which~$u(x_o)=0$,
then, by Lemma~\ref{positive}, we see that~$u$ vanishes identically.
In particular, by~\eqref{45:A},
\begin{equation}\label{0o56}
0={\mathcal{F}}(0)= {\mathcal{F}}(u)\le {\mathcal{F}}(\epsilon),\end{equation}
where~$\epsilon>0$ is a fixed constant. On the other hand,
$$ {\mathcal{F}}(\epsilon)=
\frac{c_1\epsilon^3}{3}-\frac{c_2\epsilon^2}{2}
-\frac{\tau\epsilon^2}{2},$$
where
$$ c_1:= \int_Q \mu\,dx \qquad{\mbox{ and }}
c_2:= \int_Q \sigma\,dx.$$
Notice that~$c_3:=\frac{c_2}{2}+\frac{\tau}{2}>0$, thanks to~\eqref{KJA:AKK},
and thus~${\mathcal{F}}(\epsilon) =\frac{c_1\epsilon^3}{3} -c_3\epsilon^2<0$
for small~$\epsilon$. This is in contradiction with~\eqref{0o56}
and so it proves~\eqref{0o56:BIS}. This completes the proof of Theorem~\ref{0o56:TH}.
\end{proof}

Now we establish Theorem~\ref{0o56:TH:2} via some algebraic and analytical
identities.

\begin{proof}[Proof of Theorem~\ref{0o56:TH:2}] Let~$Q$ be as in~\eqref{defQ}. We define
\begin{equation}\label{JAQ:A}
m:= \int_Q u(x)\,dx \;{\mbox{ and }}\; v(x):=u(x)-m.\end{equation}
Notice that
\begin{equation}\label{JAQ:aaA:2}
m>0,\end{equation}
due to the sign of~$u$, and
\begin{equation}\label{M:nu}
\int_Q v(x)\,dx=0.\end{equation}
Also, since~$u$ is periodic, there exists a minimal point~$x_o$,
that is
\begin{equation}\label{U:lak}
u(x_o)=\min_{Q} u=\min_{\R^n} u.
\end{equation}
Thus, since~$u$ and $v$ differ by a constant, it follows that
$$ v(x_o)=\min_{Q} v=\min_{\R^n} v.$$
This and~\eqref{M:nu} give that
\begin{equation}\label{M:nu:2}
0=\int_Q v(x)\,dx\ge v(x_o).\end{equation}
Now we point out that, for any~$y\in\R^n$,
\begin{equation}\label{M:nu:20}
\int_Q u(x+y)\,dx=m,\end{equation}
due to~\eqref{JAQ:A} and the periodicity of~$u$.
Therefore, if we
fix~$\delta>0$, we see that
$$ \int_Q \left[
\int_{\R^n\setminus B_\delta} \frac{u(x+y)+u(x-y)-2u(x)}{|y|^{n+2s}}\,dy
\right]\,dx
=\int_{\R^n\setminus B_\delta} \frac{m+m-2m}{|y|^{n+2s}}\,dy=0$$
and so, by taking~$\delta\to0$,
\begin{equation}\label{M:nu:3}
\int_Q (-\Delta)^s u(x)\,dx=0.\end{equation}
Moreover, using again~\eqref{M:nu:20}, we find that
$$ \int_Q (J*u)(x)\,dx 
=\int_Q \left[ \int_{\R^n} J(y)\,u(x-y)\,dy\right]\,dx
=m \int_{\R^n} J(y)\,dy=m.$$
Using this, \eqref{M:nu},
\eqref{M:nu:3} and the equation for~$u$, we conclude that
\begin{equation*}
\begin{split}
& 0 = \int_Q (-\Delta)^s u(x)\,dx =
\int_Q \Big(
(\sigma-\mu u)u+\tau (J*u)\Big)\,dx
=\sigma m-\mu\int_Q u^2\,dx+\tau m\\
&\qquad=
\sigma m-\mu\int_Q \big(v^2+m^2 +2mv\big)\,dx+\tau m
=\sigma m -\mu\int_Q v^2\,dx -\mu m^2 +\tau m.\end{split}\end{equation*}
This says that
\begin{equation}\label{01HAJ:A}
\mu\int_Q v^2\,dx = m\,(\sigma+\tau -\mu m).
\end{equation}
Now, we observe that, 
$$ (-\Delta)^s u(x_o)\le0,$$
thanks to~\eqref{U:lak}. 

In addition, from~\eqref{U:lak} we also deduce that
$$ (J*u)(x) =\int_{\R^n} J(y)\,u(x-y)\,dy \ge \int_{\R^n} J(y)\,u(x_o)\,dy
=u(x_o),$$
for every~$x\in\R^n$. 
Hence, we
compute the equation at~$x_o$ and we find that
\begin{eqnarray*} 
&& 0\ge (-\Delta)^s u(x_o) = (\sigma-\mu u(x_o))u(x_o)+\tau (J*u)(x_o)
\\ &&\qquad \ge (\sigma-\mu u(x_o))u(x_o)+\tau u(x_o)=
u(x_o) \,\big( \sigma+\tau-\mu u(x_o)\big).\end{eqnarray*}
Therefore, since~$u(x_o)>0$, we conclude that
$$ \sigma+\tau-\mu u(x_o)\le 0$$
and then
$$ \sigma+\tau -\mu m \le \mu \big( u(x_o)-m\big) =\mu \,v(x_o).$$
We insert this into~\eqref{01HAJ:A} and we recall~\eqref{JAQ:aaA:2},
in order to obtain that
$$ \mu\int_Q v^2\,dx = m\,(\sigma+\tau -\mu m)\le m\mu\,v(x_o).$$
Thus, by~\eqref{M:nu:2},
$$ \mu\int_Q v^2\,dx\le0,$$
which implies that~$v$ vanishes identically.
Accordingly, by~\eqref{JAQ:A}, we obtain that~$u$
is constant and constantly equal to~$m$.
We insert this information into the equation and we obtain that
$$ 0= (\sigma-\mu m)m+\tau (J*m) = (\sigma-\mu m)m+\tau m=
(\sigma+\tau-\mu m)m.$$
Recalling~\eqref{JAQ:aaA:2}, we then obtain that~$\sigma+\tau-\mu m=0$
and so~$m=(\sigma+\tau)/\mu$, as desired.
\end{proof}

\section{A transmission problem and proof of Theorems~\ref{TR:MOD:MI}
and~\ref{COP:TP}}\label{TRA:SEC}

Now we consider the transmission problem introduced in~\eqref{TP}
and we prove the existence of minimizers.

\begin{proof}[Proof of Theorem~\ref{TR:MOD:MI}]
We let~$u_j$ be a minimizing sequence. Using
the Young Inequality 
$$ ab \le \frac{a^p}{p} + \frac{(p-1)\,b^{\frac{p}{p-1}}}{p}$$
with exponents~$p=3/2$, $a=4^{-\frac23}\mu^{\frac23}u^2$ and $
b=2^{\frac13}\mu^{-\frac23}\sigma$, 
we see that
$$ \frac{\sigma \,u^2}{2} \le \frac{\mu\,|u|^3}{6}+
\frac{2\sigma^3}{3\mu^2}.$$
As a consequence,
\begin{eqnarray*}
&& 0={\mathcal{T}}(0)\ge {\mathcal{T}}(u_j)\\
&&\qquad\ge
\frac12 \int_{\Omega_1}|\nabla u_j|^2\,dx
+\frac{s\,(1-s)}{2}\iint_{\Omega_2\times\Omega_2}
\frac{|u_j(x)-u_j(y)|^2}{|x-y|^{n+2s}}\,dx\,dy\\
&&\qquad\qquad+\sum_{i=1}^2
\frac{\nu_i\,s_i\,(1-s_i)}{2}\iint_{\Omega_i\times(\R^n\setminus\Omega_i)}
\frac{|u_j(x)-u_j(y)|^2}{|x-y|^{n+2s_i}}\,dx\,dy
+\int_\Omega \frac{\mu\,|u_j|^3}{6} - 
c_o,\end{eqnarray*}
where
$$ c_o :=\int_\Omega
\frac{2\sigma^3}{3\mu^2}\,dx.$$
In particular,
$$ \frac{\mu_o}{6}
\int_\Omega |u_j|^3\,dx\le c_o,$$
which gives a uniform bound in~$j$ of~$\|u_j\|_{L^2(\Omega)}$.
Also,
$$ \frac12 \int_{\Omega_1}|\nabla u_j|^2\,dx
+\frac{s\,(1-s)}{2}\iint_{\Omega_2\times\Omega_2}
\frac{|u_j(x)-u_j(y)|^2}{|x-y|^{n+2s}}\,dx\,dy\le c_o,$$
therefore, by compactness (see e.g. Theorem~7.1
in~\cite{guida}), we find that, up to a subsequence, $u_j\to u$
in~$L^2(\Omega)$ and
a.e. in~$\Omega_1\cup\Omega_2$, with~$\nabla u_j$ converging
to~$\nabla u$ weakly in~$L^2(\Omega_1)$, for some function~$u$
vanishing outside~$\Omega$. {F}rom this, the desired result follows.
\end{proof}

The following is a maximum principle related to
the transmission problem~\eqref{TP}:

\begin{lemma}\label{MP:TP}
Let~$u$ be
a nonnegative solution of~\eqref{EQ:TP}.
Then either~$u>0$ in~$\Omega_1\cup\Omega_2$ or it vanishes identically.
\end{lemma}

\begin{proof} Assume that~$u$ vanishes somewhere
in~$\Omega_1\cup\Omega_2$. We claim that
\begin{equation}\label{CASE.1}
{\mbox{if $u$ vanishes somewhere in~$\Omega_2$, then it vanishes identically
in $\Omega_1\cup(\R^n\setminus\Omega_2)$.}}
\end{equation}
To prove this, we suppose that~$u(\bar x)=0$, for some~$\bar x\in\Omega_2$.
Then~$\bar x$ minimizes~$u$ and so
$$ PV\,\int_{\Omega_2}\frac{u(\bar x)-u(y)}{|\bar x-y|^{n+2s}}\,dy\le0,$$
$$ \int_{\Omega_1}
\frac{u(\bar x)-u(y)}{|\bar x-y|^{n+2s_1}}\,dy\le0\qquad{\mbox{ and }}\qquad
\int_{\R^n\setminus\Omega_2}
\frac{u(\bar x)-u(y)}{|\bar x-y|^{n+2s_2}}\,dy\le0.$$
These inequalities and~\eqref{EQ:TP} imply that indeed
$$ \int_{\Omega_1}
\frac{u(\bar x)-u(y)}{|\bar x-y|^{n+2s_1}}\,dy=0\qquad{\mbox{ and }}\qquad
\int_{\R^n\setminus\Omega_2}
\frac{u(\bar x)-u(y)}{|\bar x-y|^{n+2s_2}}\,dy=0,$$
and this says that~$u(y)=u(\bar x)=0$ in the whole of~$\Omega_1\cup(\R^n
\setminus\Omega_2)$,
thus proving~\eqref{CASE.1}.

Now we show that
\begin{equation}\label{CASE.2}
{\mbox{if $u$ vanishes somewhere in~$\Omega_1$, then it vanishes identically
in $\Omega_2\cup(\R^n\setminus\Omega_1)$.}}
\end{equation}
To this end, let~$x_o\in\Omega_1$
such that~$u(x_o)=0$. In particular, $x_o$ minimizes~$u$,
therefore~$ \Delta u(x_o)\ge0$, 
$$ \int_{\R^n\setminus \Omega_1}
\frac{u(x_o)-u(y)}{|x_o-y|^{n+2s_1}}\,dy\le0\qquad{\mbox{ and }}\qquad
\int_{\Omega_2}
\frac{u(x_o)-u(y)}{|x_o-y|^{n+2s_2}}\,dy \le0 .$$
These inequalities and~\eqref{EQ:TP} imply that
$$ \int_{\R^n\setminus \Omega_1}
\frac{u(x_o)-u(y)}{|x_o-y|^{n+2s_1}}\,dy=0\qquad{\mbox{ and }}\qquad
\int_{\Omega_2}
\frac{u(x_o)-u(y)}{|x_o-y|^{n+2s_2}}\,dy =0 .$$
In consequence of these equalities, we conclude that~$u(y)=u(x_o)=0$
for any~$y\in (\R^n\setminus \Omega_1)\cup\Omega_2$,
and this establishes~\eqref{CASE.2}.

Now suppose that~$u$ vanishes somewhere in~$\Omega_1$
(resp.~$\Omega_2$). Then, by~\eqref{CASE.2} (resp.,~\eqref{CASE.1}),
we know that~$u$ vanishes identically
in $\Omega_2\cup(\R^n\setminus\Omega_1)$
(resp., in $\Omega_1\cup(\R^n\setminus\Omega_2)$).
Accordingly, by~\eqref{CASE.1} (resp.,~\eqref{CASE.2}),
we obtain that~$u$ vanishes identically
in $\Omega_1\cup(\R^n\setminus\Omega_2)$
(resp., in $\Omega_2\cup(\R^n\setminus\Omega_1)$).
All in all, we find that~$u$ vanishes identically
in~$\Omega_2\cup(\R^n\setminus\Omega_1)\cup
\Omega_1\cup(\R^n\setminus\Omega_2)=\R^n$, as desired.
\end{proof}

Now we establish the results related to the spectral analysis
of the transmission problem~\eqref{TP}:

\begin{proof}[Proof of Theorem~\ref{COP:TP}]
We let~$e_\star$ be the first eigenfunction of the problem, i.e.
the minimizer which attains the infimum in~\eqref{89RFUSjKaA}.
That such minimum is attained follows by a compactness
argument, as the one in the proof of Theorem~\ref{TR:MOD:MI}.
By construction,
\begin{eqnarray*}
&&\int_{\Omega_1} \nabla e_\star\cdot\nabla\phi\,dx
+s\,(1-s)\iint_{\Omega_2\times\Omega_2}
\frac{(e_\star(x)-e_\star(y))(\phi(x)-\phi(y))}{|x-y|^{n+2s}}\,dx\,dy
\\ &&\qquad+\sum_{i=1}^2
\nu_i\,s_i\,(1-s_i)\iint_{
\Omega_i\times(\R^n\setminus\Omega_i)}
\frac{(e_\star(x)-e_\star(y))(\phi(x)-\phi(y))}{|x-y|^{n+2s_i}}\,dx\,dy
=\lambda_\star(\Omega)\,
\int_\Omega e_\star\,\phi\,dx\end{eqnarray*}
for any test function~$\phi$, and so
\begin{equation}\label{KAJ:PA678}
\begin{split}
&\int_{\Omega_1} |\nabla e_\star|^2\,dx
+s\,(1-s)\iint_{\Omega_2\times\Omega_2}
\frac{|e_\star(x)-e_\star(y)|^2}{|x-y|^{n+2s}}\,dx\,dy
\\ &\qquad+\sum_{i=1}^2
\nu_i\,s_i\,(1-s_i)\iint_{\Omega_i\times(\R^n\setminus\Omega_i)}
\frac{|e_\star(x)-e_\star(y)|^2}{|x-y|^{n+2s_i}}\,dx\,dy
=\lambda_\star(\Omega)\,\int_\Omega |e_\star|^2\,dx.\end{split}\end{equation}
Also, we may assume that~$e_\star\ge0$, since taking the absolute
value of a candidate may only decrease the energy, and in fact
\begin{equation}\label{E0MP:PT}
{\mbox{$e_\star>0$
in~$\Omega_1\cup\Omega_2$,}}\end{equation}
thanks to the maximum principle in Lemma~\ref{MP:TP}.

Given~$M>0$, we set
$$ e_M(x):= \left\{
\begin{matrix}
e_\star(x) & {\mbox{ if }} e_\star(x)<M,\\
M & {\mbox{ if }} e_\star(x)\ge M.
\end{matrix}
\right.$$
By the Fatou Lemma,
$$ \liminf_{M\to+\infty} \int_\Omega \sigma\, e_M^2\,dx
\ge \int_\Omega \sigma \, e_\star^2\,dx,$$
and therefore
\begin{equation}\label{FAT:TP}
\liminf_{M\to+\infty} \int_\Omega \sigma\, e_M^2 - \lambda_\star(\Omega)\,e^2_\star\,dx
\ge 
\int_\Omega (\sigma- \lambda_\star(\Omega))\,e^2_\star\,dx =: c_\star.
\end{equation}
After these considerations, we proceed with the proof of Theorem~\ref{COP:TP}.

First, we suppose that~$\sup_{\Omega}\sigma\le\lambda_\star(\Omega)$.
We aim to show that all solutions of~\eqref{EQ:TP}
are trivial. Assume, by contradiction, that there exists
a nontrivial solution~$u$. Then, by Lemma~\ref{MP:TP},
we know that~$u>0$ in~$\Omega_1\cup\Omega_2$.

Now, we write the weak formulation of~\eqref{EQ:TP} as 
\begin{eqnarray*} 
&&\int_{\Omega_1}\nabla u\cdot\nabla\phi\,dx 
+s\,(1-s)\iint_{\Omega_2\times\Omega_2} 
\frac{(u(x)-u(y))(\phi(x)-\phi(y))}{|x-y|^{n+2s}}\,dx\,dy\\ 
&&\qquad+\sum_{i=1}^2 
{\nu_i\,s_i\,(1-s_i)}\iint_{\Omega_i\times(\R^n\setminus\Omega_i)}
\frac{(u(x)-u(y))(\phi(x)-\phi(y))}{|x-y|^{n+2s_i}}\,dx\,dy \\&&\qquad
+\int_\Omega {\mu\,u^2\,\phi} - {\sigma \,u^2}\,dx\;=\;0
,\end{eqnarray*} for any test function~$\phi$, and we choose~$\phi:=u$. 
Hence, we find that 
\begin{eqnarray*} &&\int_{\Omega_1}|\nabla u|^2\,dx 
+s\,(1-s)\iint_{\Omega_2\times\Omega_2} 
\frac{|u(x)-u(y)|^2}{|x-y|^{n+2s}}\,dx\,dy\\ &&\qquad+\sum_{i=1}^2 
\nu_i\,s_i\,(1-s_i)\iint_{
\Omega_i\times(\R^n\setminus\Omega_i)}
\frac{|u(x)-u(y)|^2}{|x-y|^{n+2s_i}}\,dx\,dy \\&&\qquad
+\int_\Omega {\mu\,|u|^3} - {\sigma \,u^2}\,dx
\;=\;0
.\end{eqnarray*}
As a consequence,
\begin{eqnarray*}
\lambda_\star(\Omega) &\le&
\|u\|_{L^2(\Omega)}^{-2}\,{\mathcal{T}}_o(u)\\
&=&
\|u\|_{L^2(\Omega)}^{-2}\,
\int_\Omega {\sigma \,u^2}-{\mu\,|u|^3} \,dx
\\ &<&\|u\|_{L^2(\Omega)}^{-2}\, \int_\Omega {\sigma \,u^2}\,dx
\\ &\le& \|u\|_{L^2(\Omega)}^{-2}\, \int_\Omega {\lambda_\star(\Omega)
\,u^2}\,dx\\
&=&
\lambda_\star(\Omega),
\end{eqnarray*}
which is a contradiction.
This establishes the first claim in
Theorem~\ref{COP:TP}, so we can now focus on the second claim.
To this goal, we now assume that
$\inf_{\Omega}\sigma\ge\lambda_\star(\Omega)$ with strict
inequality on a set of positive measure. 
Therefore, recalling~\eqref{E0MP:PT},
we have that, in this case,
$$ c_\star>0 $$
and so, in light of~\eqref{FAT:TP}, we can fix~$M_\star$
sufficiently large such that, for any~$M\ge M_\star$,
$$ \int_\Omega \sigma\, e_M^2 - \lambda_\star(\Omega)\,e^2_\star\,dx
\ge \frac{c_\star}{2} >0.$$
So, from now on, we can fix~$M=M_\star$, and the inequality above
holds true. In consequence of these observations
and recalling~\eqref{KAJ:PA678}, we have that
\begin{eqnarray*}
{\mathcal{T}}_o(e_M)&\le&
\int_{\Omega_1}|\nabla e_\star|^2\,dx
+s\,(1-s)\iint_{\Omega_2\times\Omega_2}
\frac{|e_\star(x)-e_\star(y)|^2}{|x-y|^{n+2s}}\,dx\,dy
\\ &&\qquad+\sum_{i=1}^2
\nu_i\,s_i\,(1-s_i)\iint_{
\Omega_i\times(\R^n\setminus\Omega_i)}
\frac{|e_\star(x)-e_\star(y)|^2}{|x-y|^{n+2s_i}}\,dx\,dy \\
&=& \lambda_\star(\Omega)\,\int_\Omega |e_\star|^2\,dx\\
&\le& -\frac{c_\star}{2} + \int_\Omega \sigma\, e_M^2\,dx.
\end{eqnarray*}
Accordingly, for any~$\epsilon>0$,
\begin{eqnarray*}
{\mathcal{T}}(\epsilon e_M)&=&
\epsilon^2\,{\mathcal{T}}_o(e_M)
+\int_\Omega \epsilon^3\,\frac{\mu\,|e_M|^3}{3} - 
\epsilon^2\,\frac{\sigma \,e_M^2}{2}\,dx \\
&\le& -\frac{c_\star\,\epsilon^2}{2} 
+\epsilon^3\int_\Omega \frac{\mu\,|e_\star|^3}{3},\end{eqnarray*}
which is negative if~$\epsilon$ is suitably small.
As a consequence, ${\mathcal{T}}(\epsilon e_M)<0={\mathcal{T}}(0)$,
which implies that the trivial function is not a minimizer.

This says that the minimizer does not vanish identically,
and so it is positive in~$\Omega_1\cup\Omega_2$,
in light of Lemma~\ref{MP:TP}.
This completes the proof of Theorem~\ref{COP:TP}.
\end{proof}

\section*{Acknowledgements}
Part of this work was carried out while Serena Dipierro
and Enrico Valdinoci were visiting the {\it Department of Mathematics}
and the {\it Institute for Computational Engineering and Sciences} 
of the {\it University of Texas at Austin}, which
they wish to thank for the support and warm hospitality. 

This work has been supported by NSF grant DMS-1160802, 
the Alexander von Humboldt Foundation, the
ERC grant 277749 {\it E.P.S.I.L.O.N.} ``Elliptic
Pde's and Symmetry of Interfaces and Layers for Odd Nonlinearities'',
and the PRIN grant 201274FYK7
``Aspetti variazionali e
perturbativi nei problemi differenziali nonlineari''.

\bibliographystyle{alpha}
\newcommand{\etalchar}[1]{$^{#1}$}

\end{document}